\documentclass[11pt]{article}

\usepackage{amsmath, amssymb, amsthm, mathtools}       
\usepackage{bm, enumerate, mathrsfs, tensor, upgreek}  
\usepackage{indentfirst, emptypage, lmodern, setspace} 
\usepackage{caption, figsize, multirow, subfigure}     
\usepackage[greek,english]{babel}                      
\usepackage{csquotes}
\usepackage{authblk}
\usepackage{hyperref}
\usepackage[dvipsnames]{xcolor}
\usepackage{tikz}
\usepackage{adjustbox}
\usetikzlibrary{fadings, decorations.pathreplacing}

\definecolor{coprochroma}{RGB}{100,84,3}

\usepackage[linesnumbered,ruled,vlined]{algorithm2e}

\usepackage[margin=0.7in]{geometry}
\theoremstyle{plain}
\newtheorem{thm}{Theorem}[section]
\theoremstyle{definition}

\newtheorem{rem}[thm]{Remark}
\newtheorem{lem}[thm]{Lemma}
\newtheorem{prop}[thm]{Proposition}
\newtheorem{hyp}[thm]{Hypothesis}

\newcommand{\dd}[1]{\mathrm{d}\mathrm{#1}}
\newcommand{\stkout}[1]{\ifmmode\text{\sout{\ensuremath{#1}}}\else\sout{#1}\fi}

\newcommand{\HH}{{\mathbb{H}^2}}
\newcommand{\HHD}{{\mathbb{H}^2_D}}
\newcommand{\RR}{\mathbb{R}}
\newcommand{\NN}{\mathbb{N}}
\newcommand{\Z}{\mathbb{Z}}
\newcommand{\x}{\mathbf{x}}
\newcommand{\y}{\mathbf{y}}

\newcommand{\CC}{\mathcal{C}}

\newcommand{\nord}{\mathcal{N}}
\newcommand{\sud}{\mathcal{S}}
\newcommand{\est}{\mathcal{E}}
\newcommand{\vest}{\mathcal{W}}
\newcommand{\punct}{\mathcal{P}}

\numberwithin{equation}{section}

\usepackage[backend=biber,giveninits=true,style=numeric,citestyle=numeric, doi=false,isbn=false,url=false,eprint=false,maxbibnames=99]{biblatex}
\bibliography{numerical-hyperbolic}

\title{
Discrete Laplacians on the hyperbolic space -- a comparative study
}
\author[1,2,\thanks{E-mail: \texttt{mihai.bucataru@fmi.unibuc.ro;} \href{https://orcid.org/0000-0002-6503-2084}{ORCID ID: 0000-0002-6503-2084}}]
{Mihai Bucataru}
\author[1,3,\thanks{Corresponding author. E-mails: \texttt{dmanea28@gmail.com}, \texttt{dmanea@imar.ro}; \href{https://orcid.org/0000-0003-4085-226X}{ORCID ID: 0000-0003-4085-226X}. D. M. was partially supported by CNCS-UEFISCDI Romania, Grant no. 0794/2020
``Spectral Methods in Hyperbolic Geometry'' PN-III-P4-ID-PCE-2020-0794 and by the PNRR-III-C9-2023-I8 grant
CF 149/31.07.2023 ``Conformal Aspects of Geometry and Dynamics''.}]
{Dragoș Manea}
\affil[1]{\small Department of Mathematics, Faculty of Mathematics and Computer Science, University of Bucharest, 14~Academiei, 010014~Bucharest, Romania}
\affil[2]{\small ``Gheorghe Mihoc~--~Caius Iacob" Institute of Mathematical Statistics and Applied Mathematics of the Romanian Academy, 13~Calea 13 Septembrie, 050711~Bucharest, Romania}
\affil[3]{\small ``Simion Stoilow" Institute of Mathematics of the Romanian Academy, 21~Calea Grivi\c tei, 010702~Bucharest, Romania}
\date{}

\begin{document}
\maketitle

\begin{abstract}\noindent

This paper is concerned with the construction of discrete counterparts of the Laplace-Beltrami operator on Riemannian manifolds that can be effectively used in the numerical solution of partial differential equations. Since existing constructions often lack rigorous convergence guarantees or imply a significant computational effort, we focus on designing operators that are both computationally feasible and supported by convergence results.

We consider as a starting point the two-dimensional hyperbolic space $\mathbb{H}^2$, one of the simplest non-Euclidean settings, and develop two variants of discrete finite-difference operator tailored to this constant negatively curved space, both serving as approximations to the (continuous) Laplace-Beltrami operator within the $\mathrm{L}^2$ framework. 
    
    We prove that the discrete heat equation 
    associated with both operators mentioned above exhibits stability and converges towards the continuous heat-Beltrami Cauchy problem on $\mathbb{H}^2$.
    Moreover, using techniques inspired from the sharp analysis of discrete functional inequalities, we prove that the solutions of the discrete heat equations corresponding to both variants of discrete Laplacian exhibit an exponential decay asymptotically equal to the one induced by the Poincar\'e inequality on $\HH$. 
    
    Eventually, we illustrate that a discrete Laplacian specifically designed for the geometry of the hyperbolic space yields a more precise approximation and offers advantages from both theoretical and computational perspectives. Furthermore, this discrete operator can be effectively generalized to the three-dimensional hyperbolic space.
\end{abstract}
\hspace*{0.8cm} \textbf{2020 Mathematics Subject Classification:} 65M06, 65M15, 58J35, 58J60, 35K08. \\[10pt]
\hspace*{0.8cm} \textbf{Keywords:} Finite difference method (FDM), PDEs on Riemannian manifolds,\\
\hspace*{0.8cm} Numerical methods on manifolds, Geometry-tailored numerical grid, Discrete functional inequalities.
\section{Introduction}
In this work, we aim to construct a discrete Laplacian on hyperbolic spaces that is explicit, computationally feasible, and provably convergent when used in PDE solvers. As a first step, we focus on one of the simplest non-trivial (i.e., non-Euclidean) settings, namely the two-dimensional hyperbolic space $\mathbb{H}^2$, while designing the method in such a way that it naturally extends to higher dimensional hyperbolic spaces.

Our goal is therefore not only to define a discrete Laplacian on the curved space $\mathbb{H}^2$, but to ensure that this operator fulfills its fundamental role: to accurately approximate partially diferential equations, most notably the heat equation (possibly with source terms). That is, we seek a discrete operator that leads to a numerical scheme converging to the solution of the heat-Beltrami equation on hyperbolic space:
\begin{eqnarray} \label{eq:heat-hyperbolic-source}
\left\{
\begin{array}{ll}
\partial_t u(t, \mathbf{x}) = \Delta_g u(t,\mathbf{x})+f(t,\x)\,,
& t \in (0, T],\; \mathbf{x} \in \HH\,, \\[4pt]
\phantom{\partial_t} u(0, \mathbf{x}) = u_0(\mathbf{x})\,,
& \mathbf{x} \in \HH\,,
\end{array}
\right.
\end{eqnarray}
where $T>0$ is any fixed time, $u_0 \in {\rm L}^2(\HH)$ is the initial data, $\Delta_g$ is the Laplace-Betrami operator on $\HH$ and $f\in C([0,T],\mathrm{L}^2(\HH))$ is the source term.
Having defined our purpose, let us see what are the main ingredients used to build a discrete Laplacian:
\begin{enumerate}[i)]
    \item \label{discrete-lapl-first-aim}First, we need a grid, that is a discrete set of points well arranged in space, somehow evenly dispersed in $\HH$, whilst bearing in mind that it must be efficiently implemented on a computer. Notice that the curvature may influence the distribution of the grid points.
    \item Then, we need a means to transfer information between the entire continuous space and the discrete grid. Since we focus on the ${\rm L}^2$ settings, the transfer procedure should be compatible with both discrete and continuous ${\rm L}^2$ norms.
    \item Once these tasks are accomplished, we can fashion our discrete Laplacian. Much like in the Euclidean plane, it emerges as a linear combination of the values of our function in five adjacent points, their weights meticulously mirroring the even diffusion of heat across the curved space we study.
    \item Subsequently, we justify the stability of both semi-discrete and fully discrete numerical schemes corresponding to the discrete Laplacian and we assess the order of error relative to the continuous solution of the heat-Beltrami problem with source.
    \item \label{discrete-lapl-last-aim} Finally, our discrete framework grants us an additional perk -- the restoration of the ${\rm L}^2$ exponential stability akin to the continuous homogeneous heat problem. To achieve this, we prove a discrete Poincaré inequality tailored for our Laplace operator. This inequality resembles the Poincaré inequality in the entire space $\HH$, together with its optimal constant derived from the inherent negative curvature.
\end{enumerate}

A good starting point is to focus on one of the models of the hyperbolic space that is most suited to a finite-difference approach. Among the isometrically equivalent models briefly summarised in \cite{loustau2021hyperbolicgeometry} we have chosen the Poincar\' e upper half-plane $\HH \simeq \RR^2_+ \coloneqq \{\mathbf{x} = (x_1, x_2) |\; x_1, x_2 \in \mathbb{R},\; x_2 > 0\}$ , in which the hyperbolic metric reads
\begin{equation}\label{eq:cont01}
    g_\x\coloneqq \dfrac{(dx_1) ^ 2 + (dx_2) ^ 2}{x_2 ^ 2}.
\end{equation}

Given that the model's support set in this instance is a subset of $\RR^2$, the initial straightforward yet effective approach involves employing the standard uniform Euclidean grid confined to the upper half-plane. However, a pertinent question naturally arises: \textit{Could an alternative approach that significantly considers the geometry of the space yield more precise results while concurrently enhancing resource utilisation?}

In the present paper, we aim to address the above question by formulating two versions of the discrete Laplacian. One relies on the uniform Euclidean grid, while the other is tailored to the specificity of curved space. Both fit into the framework \ref{discrete-lapl-first-aim})-\ref{discrete-lapl-last-aim}), however, as anticipated, the latter variant yields more precise results, all the while optimising memory usage for grid construction and associated functions.

Our approach finds its place in a vast series of attempts to numerically approximate the solutions of differential equations on Riemannian manifolds. Starting from zero-order equations \cite{AlvarezBolte2008}, going through ODEs describing curves on manifolds \cite{Fiori2017,BarettGarke2019,GarkeNurnberg2021,FioriCervigni2022} and arriving to dynamical systems and other PDEs \cite{Holst2001,AmorimBenArtzi2005,Giesselmann2009}, various methods were used to tackle this approximation problem. Especially for ODEs, one of the most preferred method is to compute the solution iteratively using normal geodesic coordinates around the current point, whilst, in the case of PDEs, finite element \cite{Holst2001}, finite volume \cite{AmorimBenArtzi2005,Giesselmann2009} and Monte Carlo \cite{CruzeiroMalliavin2006} are employed. However, none of the methods enumerated above does literally construct a discrete counterpart of the differential operator they analyse.

While several approaches to discrete Laplacians on Riemannian manifolds exist in the literature, they present important limitations. For instance, in \cite{Jiang2024}, the discrete operator is constructed using weights computed locally at each point, leading to increased computational cost. Moreover, convergence is established only on compact manifolds and holds with a certain probability, as it relies on randomly sampled point clouds. The method further requires a sufficiently large number of nearest neighbours for the stencil, chosen to satisfy an a posteriori condition. By contrast, the approach developed in this paper yields a globally defined discrete Laplacian with deterministic structure, which does not rely on random sampling or adaptive neighbour selection. Most importantly, we establish rigorous convergence results for the associated heat equation on the hyperbolic space.

On the other hand, we mention \cite{BuragoIvanov2014} which recovers the eigenvalues of the Laplace-Beltrami operator on a manifold by constructing a discrete operator on an embedded graph and \cite{Izmestiev2025}, where a definition of discrete Laplacians is provided for general hyperbolic surfaces. However, key numerical aspects -- such as order of accuracy, stability, and convergence in the context of PDEs -- are not addressed.

Another category of scientific literature our approach aims to extend refers to integral inequalities on Riemannian manifolds, especially negatively curved ones. As seen in the point \ref{discrete-lapl-last-aim}) above, integral inequalities serve, among others, to derive well-posedness, stability and decay properties of PDEs and also to estimate eigenvalues of differential operators. Far from pretending an exhaustive survey, we mention the works \cite{MugelliTalenti1997,MugelliTalenti1998,sharpPoincare2018,NgoNguyen2019} pertaining to sharp Sobolev and Poincar\' e - type inequalities on the hyperbolic space, \cite{BerchioGanguly2020} for the Hardy inequality on the same space and \cite{Kristaly2022} for estimates of the first eigenvalue of the Laplacian in negatively curved spaces.

The literature pertaining to the heat equation on the hyperbolic space of any dimension is also well-developed. The corresponding heat kernel was computed and estimated uniformly in \cite{DaviesMandouvalos,GrigoryanHeatKernelOnHyperbolic} and the asymptotic behaviour and the existence of long-time profiles for the heat equation on the hyperbolic space was analysed in \cite{Vazquez2022}. We refer to \cite{AnkerPapageorgiou2023} for the study of the long-time asymptotic behaviour of the heat equation in a more general setting, that of symmetric spaces. Moreover, the approximation of the solutions of the heat equation on the hyperbolic space with a non-local problem was studied in \cite{BandleGonzalez2018}, whereas the papers \cite{Banica2007} and \cite{Tataru2001} are dedicated to the Schr\" odinger and wave equations on the hyperbolic space, respectively. 

One instrument that plays an important role role in the study of the Laplace-Beltrami operator on the hyperbolic space is the so-called \emph{Fourier-Helgason transform}, that can be defined in the more general setting of symmetric spaces (see, for example \cite{Helgason2008,Terras2013}). This integral transform was successfully employed, for example, in the aforementioned works \cite{Tataru2001,Banica2007, BandleGonzalez2018}, whereas in \cite{Pesenson2009} the authors introduced a discrete version of it
and managed to build approximations of ${\rm L}^2$ functions with discrete counterparts. However, due to the rather cumbersome form of this transform, we prefer to use more direct methods in the construction of our discrete Laplace operators.

Another aspect worth mentioning is that, since in the present paper we aim to solve the heat equation posed in the whole space $\HH$ on a computer with limited resources -- finite memory and processing speed -- we need to restrict ourselves to a bounded subset of $\HH$, while taking care of the boundary conditions so that the reduced problem can still approximate the continuous problem on the entire space. In this sense, we simultaneously employ two types of refinement of the approximation: along with the reduction of the parameter $h$ accounting for the step size of the finite difference grid itself, we enlarge the bounded domain
\begin{equation}\label{Intro:eq:cine-PLM-i-HD}
    \HHD\coloneqq [-D,D]\times \left [\frac 1 D,D\right] \subset \HH\,,
\end{equation}
on which we pose the actual discrete heat equation. In particular, in our analysis, we choose the variable $D$, describing the size of the discrete domain, to increase polynomially with respect to $\frac 1 h$:
\begin{equation}\label{Intro:eq:cine-PLM-i-D}
D_{h,\gamma,\zeta}\coloneqq \zeta h^{-\gamma}, \text{ with } \zeta > 2\,, \gamma >0.
\end{equation}
Fortunately, since the heat equation in $\HH$ decays quickly enough as the spatial variable tends to infinity, we can impose zero Dirichlet \emph{artificial boundary conditions} to the reduced discrete problem and obtain, for a properly chosen source term, convergence to the solution of \eqref{eq:heat-hyperbolic-source}. Refer to Section \ref{section:First-Discrete-Laplacian} for more details on the aforementioned spatial approximation. 

Furthermore, within each of the discrete frameworks defined by our discrete Laplace operators, we establish a Poincaré inequality that closely mirrors its continuous analogue on the hyperbolic space 
$\HH$, and we explicitly compute the associated optimal constant. These results not only reinforce the fidelity of our discrete model in capturing key geometric and analytic features of 
$\HH$, but also position our work within the literature on discrete functional inequalities with sharp constants (see, e.g., \cite{david2022DiscreteHardyLine,beci-yehuda,david-stampac,pescaru-beci}). In this way, our contribution creates a connection between the discrete numerical approximation of operators on negatively curved spaces and the rigorous theory of discrete inequalities with optimal constants.

The paper is organised as follows. In Section~\ref{sec:preliminarii} a brief introduction to the geometry of the hyperbolic space. Several properties regarding the heat equation associated to the Laplace-Beltrami operator on the hyperbolic space, such as the exponential stability and tail control, are presented in Section~\ref{sec:continuous-case}. Further on, the construction of the first and second discrete Laplacians are presented in Sections~\ref{sec:first-lapl} and~\ref{sec:second-lapl}, respectively, together with consistency estimates of the order $\mathcal{O}(h^2)$. \\

The main results of the paper concern the convergence of order $\mathcal{O}(h^2)$ of the semi-discrete scheme associated with both variants of discrete Laplace operators defined in the aforementioned sections and the convergence of order $\mathcal{O}(h^2+\tau^2)$ of the corresponding fully-discrete Crank-Nicolson scheme. They can be consulted in Sections~\ref{sec:semi-discret} and \ref{sec:crac-stix}, respectively. This convergence order is sharp, as seen from the numerical experiments performed in Section~\ref{sec:numerics}, where we also extend the second discrete Laplacian to the 3D hyperbolic space $\mathbb{H}^3$ and use it to solve the stationary heat equation. We finish our paper by drawing some conclusions and suggesting further research directions in Section~\ref{sec:conclusion}.

\section{Preliminaries on the hyperbolic space}\label{sec:preliminarii}

We start our preliminary section with a basic introduction into the geometry of the $n$-dimensional hyperbolic space $\mathbb{H}^n$, together with the differential operators that are necessary for our study. From a geometric point of view, the hyperbolic space is defined as the unique $n$-dimensional, complete, simply connected Riemannian manifold with constant sectional curvature equal to $-1$. This abstract definition makes sense, since all the manifolds satisfying the aforementioned properties are isometric. Therefore, it is enough to work on a model of the hyperbolic space, and the Riemannian properties that we obtain can be transferred to any other model through isometries.

In the sequel, we will drive our attention to the seemingly most appropriate model of the hyperbolic geometry for employing numerical methods, that is the half-space model, suggesting to the interested reader to consult \cite{loustau2021hyperbolicgeometry} for a survey of the most used models that exist in the literature. Thus, throughout this paper, $\mathbb{H}^n$ will denote the half-space model of $n$-dimensional hyperbolic geometry. Its support set is the upper half-space of $\RR^n$:
\[\mathbb{H}^n \simeq \RR^n_+\coloneqq\{\mathbf{x} = (x', x_n) |\; x'\in \RR^{n-1}, x_n > 0\}, \]
endowed with the metric:
\[g_\x\coloneqq \dfrac{(dx_1) ^ 2 + (dx_2) ^ 2+\cdots (dx_n)^2}{x_n ^ 2}.\]
We remark that, as opposed to the Euclidean metric on the half-space, the hyperbolic distances approach infinity near the base hyperplane $\partial \mathbb{H}^n\simeq \RR^{n-1}$ and decrease as $x_n$ goes to infinity. More precisely, one can compute the hyperbolic distance in this model in the following manner:
\begin{equation}\label{eq:hyp-distance}
d_{g}(\mathbf{x},\mathbf{y})= {\rm acosh}\left(1+\frac{\|\mathbf{x}-\mathbf{y}\|_e^2}{2x_n y_n}\right),
\end{equation}
where $\|\cdot\|_e$ denotes the Euclidean norm of a vector.
As an early disclosure, this behaviour of the hyperbolic distance towards the extremities of the $x_n$ axis will inspire the construction of our second grid (in Section \ref{sec:construction-second-lapl}), the one more tailored to this specific geometry.

Since it is a way of measuring distances, every Riemannian metric induces a measure (from a geometric point of view, a volume form) on the underlying manifold, together with differential operators, which are counterparts of the gradient, divergence and Laplacian operators on $\RR^n$. For a detailed discussion of those operators in a general setting, one could consult \cite[Section 2.1]{IgnatManeaMoroianu2024}. Here, we will only provide the form of the hyperbolic measure $\mu$:
\begin{align*}
\int_{\mathbb{H}^n} v(\x)\,\dd \mu(\x)=\int_{\RR^n_+} v(\x)\frac{1}{x_n^n} \,\dd \x,
\end{align*}
 and the differential operators on the particular model we are working on (following  \cite[Section 2.2]{IgnatManeaMoroianu2024}):
\begin{align*}
\nabla_g v = x_n^2 \nabla_e v, &&
 {\rm div}_g(Y)=x_n^n \,{\rm div}_e\left(\frac{1}{x_n^n} Y\right), && \Delta_g v = x_n^n\, {\rm div}_e \left(\frac{1}{x_n^{n-2}} \nabla_e v\right), 
\end{align*}
where the indices "$g$" and "$e$" indicate that the operator corresponds to the hyperbolic and Euclidean metrics, respectively. One should pay attention that we have chosen the sign of $\Delta_g$ according to the analysts' convention, for which the Laplacian is a non-positive operator.

The operator of interest in this study is the $2$-dimensional Laplace-Beltrami operator on the half-plane model of $\HH$:
\begin{equation}\label{def:laplacianH2}
    \Delta_g v (\x)= x_2^2 (\partial^2_{x_1} v+ \partial^2_{x_2} v),\text{ for }\x=(x_1,x_2)\in \HH \simeq \RR^2_+.
\end{equation}

\section{The continuous case}
\label{sec:continuous-case}

The aim of this section is to emphasize some of the properties of the heat equation associated to the Laplace-Beltrami operator on $\HH$ that will be useful for its approximation with a finite-difference numerical scheme.
\begin{hyp}\label{hyp:biserica-Grivitei}  Throughout this paper, we will impose the following regularity for the initial datum $u_0$ and the source. For the initial datum, we assume:
\begin{equation}\label{eq:biserica-Grivitei}
u_0\in \mathcal{M}\coloneqq \left\{v\in C^6(\HH) : \|v\|_{\mathcal{M}}\coloneqq \sup_{w>0}\,\sum_{\alpha=0}^6 \left\| e^{ w\,d_g(\mathrm{O},\x)} |\nabla_g^\alpha v|_g \right\|_{{\rm L}^2(\HH)} <\infty\right\},
\end{equation}
where by $\nabla_g^\alpha$ we understand the Riemannian covariant derivative applied $\alpha$ times to the function $v$ and $|\cdot|_g$ stands for the extension of the metric $g$ to the space of $(0,\alpha)$-type tensors on $T\HH$. 

The point $\mathrm{O} \in \HH$ is an arbitrary fixed point, assumed for simplicity to be $\mathrm{O} \coloneqq (0,1)\in \HH$, as the set $\mathcal{M}$ is invariant to the choice of this point. We note that by ${\rm L}^2(\HH)$ we mean the Lebesgue Hilbert space corresponding to the measure $\mu$ induced by the metric $g$ on $\HH$.

For the source term, we assume that the norm defined above is continuous is time:
\begin{equation}
f\in C([0,T],\mathcal{M}),\, \text{ for a fixed time }T>0.
\end{equation}

\end{hyp}

\begin{rem}
Since in this paper we work on the half-plane model of $\HH$, it is useful to note that Hypothesis \ref{hyp:biserica-Grivitei} is equivalent to:
\[\sup_{w>0}\sum_{\substack{\alpha,\beta\in \NN\\ \alpha+\beta\leq 6}}\left\| e^{w\, d_g(\mathrm{O},\x)} \partial_{x_1}^\alpha \partial_{x_2}^\beta v(\x)\right\|_{{\rm L}^2(\HH)} <\infty.\]
\end{rem}

\begin{rem}
The $C^6(\HH)$ regularity in Hypothesis \ref{hyp:biserica-Grivitei} is only used to prove the order of convergence $\mathcal{O}(h^2+\tau^2)$ for the discrete Crank-Nicolson scheme in Theorem \ref{thm:convergence-theta}. For all the other convergence results, it is sufficient to impose $C^4(\HH)$ regularity on $u_0$ and $f$. 

Furthermore, the exponential weighting of the $\mathrm{L}^2$ norm of the derivatives in $\|\cdot\|_{\mathcal{M}}$ is essential in order to obtain the tail decay estimates in Section \ref{sec:kernel_estimates}, which in turn lead to the approximation of the heat problem on the whole space $\HH$ with a numerical scheme on the bounded domain \eqref{Intro:eq:cine-PLM-i-HD}.
\end{rem}

\subsection{Tail control}

The following proposition describing the ${\rm L}^2$ tail control of functions in terms of weighted integrals as in \eqref{eq:biserica-Grivitei} will be used to estimate the error between the solution of the numerical scheme -- which is defined on the bounded domain $\HHD$ defined in \eqref{Intro:eq:cine-PLM-i-HD} -- and the solution of the continuous problem \eqref{eq:heat-hyperbolic-source}.
\begin{prop}\label{prop:tail-control}
Let $D \coloneqq D_{h, \gamma, \zeta}$ chosen as in~\eqref{Intro:eq:cine-PLM-i-D}, $\alpha \in \{0, 2\}$, $\beta \in \{2, 4\}$. Then, there exists a constant $C_{\gamma,\zeta}>0$ such that, for any $v\in {\rm L}^2(\HH)$ and any $h>0$,
\begin{equation}\label{eq:tail-control}
\frac{D^\alpha}{h^\beta}\|v\|_{{\rm L}^2\left(\HH\setminus\mathbb{H}^2_{D}\right)} \leq C_{\gamma,\zeta} \left\|e^{\left(2+\frac{4}{\gamma}\right) d_g(\mathrm{O},\x)} v(\x)\right\|_{{\rm L}^2(\HH)}.
\end{equation}
\end{prop}
\begin{proof}
    The particular choice of $D$ enables us to rewrite:
    \[\frac{D^\alpha}{h^\beta}=\zeta^{-\frac{\beta}{\gamma}}D^{\alpha+\frac{\beta}{\gamma}}.\]
   Further, using the formula \eqref{eq:hyp-distance} of distance function on $\HH$ we obtain that, if $\x\in \HH\setminus\mathbb{H}^2_{D}$, then $d_g(\mathrm{O},\x)\geq \log(D/2)$. In other words, 
   \[D \leq 2e^{d_g(\mathrm{O},\x)},\] 
   so we can estimate the left-hand side term in \eqref{eq:tail-control} as follows:
    \[
    \left(\frac{D^\alpha}{h^\beta}\right)^2 \int_{\HH\setminus\mathbb{H}^2_{D}} (v(\x))^2 \dd{\x} \leq 2^{2\left(\alpha+\frac{\beta}{\gamma}\right)}\zeta^{-\frac{2\beta}{\gamma}} \int_{\HH\setminus\mathbb{H}^2_{D}}e^{2\left(\alpha+\frac{\beta}{\gamma}\right) d_g(\mathrm{O},\x)} (v(\x))^2 \dd{\x}
    \]
   and obtain the conclusion.
\end{proof}

\subsection{Poincar\' e inequality and exponential stability}
\label{sec:poincare-continuous}

In contrast to the Euclidean spaces $\RR^n$, the negative curvature of the hyperbolic spaces induces a Poincar\' e-type inequality on the whole space, which, in turns, leads to the exponential decay of the solution of the homogeneous heat equation in ${\rm L}^2(\HH)$. In the particular case of the $2$-dimensional space $\HH$, the sharp Poincar\' e inequality reads as:

\begin{prop}[{\cite{sharpPoincare2018}}]
Let $v \in H^1(\HH)$. Then the following inequality holds true:
\begin{equation}\label{eq:poincare-continuous}
    \int_\HH |\nabla_g v|_g^2 \dd \mu \geq \frac{1}{4} \int_\HH |v|^2 \dd \mu
\end{equation}
and the constant $\frac{1}{4}$ cannot be improved.
\end{prop}
A direct consequence of this inequality is the following exponential ${\rm L}^2$ stability estimate for the solutions of the homogeneous heat equation on $\HH$ (i.e., the case $f\equiv 0$ and $T=\infty$). We note that, since, by definition  (see, for example, \cite[Section 2.1]{IgnatManeaMoroianu2024}), the Laplace-Beltrami operator is a non-positive self-adjoint unbounded operator on ${\rm L}^2(\HH)$, the solution of the heat equation is given by a contractions semigroup in the sense of \cite{cazenave}:
\begin{equation}\label{eq:sol-heat-eq-semigroup}
u(t,\x)=\left(e^{t\Delta_g} u_0\right)(\x),\, t\in [0,\infty), \x\in \HH.    
\end{equation}

\begin{prop}
    Let $u_0\in {\rm L}^2(\HH)$. Then, the following ${\rm L}^2$ exponential decay holds true:
    \[\left\|e^{t\Delta_g} u_0\right\|_{{\rm L}^2(\HH)} \leq e^{-\frac{t}{4}} \|u_0\|_{{\rm L}^2(\HH)}.\]
\end{prop}
\begin{proof}
    An integration by parts (refer again to \cite[Section 2.1]{IgnatManeaMoroianu2024}) in the Poincar\' e inequality \eqref{eq:poincare-continuous} implies that 
    \[-\left\langle \Delta_g e^{t\Delta_g} u_0,e^{t\Delta_g} u_0\right\rangle_{{\rm L}^2(\HH)} \geq \frac{1}{4} \left\|e^{t\Delta_g} u_0\right\|_{{\rm L}^2(\HH)}^2.\]
    Then, the desired inequality follows by {\cite[Theorems 3.1.1 and 3.2.1]{cazenave}}.
\end{proof}

\subsection{Heat kernel estimates}
\label{sec:kernel_estimates}
In addition to the study of the heat equation on $\mathbb{H}^n$ solely by means of semigroup theory, one can dive deeper into a more explicit form of the solutions via an integral kernel, similarly to the Euclidean setup. More precisely, for any $n\geq 2$, there exists a function $K_n:(0,\infty)\times [0,\infty)\to [0,\infty)$ such that the solution of the homogenous heat equation on $\mathbb{H}^n$ can be given as the following integral expression:
\begin{equation} \label{eq:heat-solution-convolution}
    \left(e^{t\Delta_g} u_0\right)(\x)= \displaystyle \int_{\mathbb{H}^n} K_n(t, d_g(\mathbf{x}, \mathbf{y}))\, u_0(\mathbf{y})\, \dd{\mu(\mathbf{y})}, \, \forall \x \in \HH.
\end{equation}
We refer to \cite{DaviesMandouvalos,GrigoryanHeatKernelOnHyperbolic} for more insight about the heat kernel on $\mathbb{H}^n$. For the purpose of our study, we will need a uniform estimate of the kernel from \cite{DaviesMandouvalos}, which is summarised in the next proposition:

\begin{prop}[{\cite[Theorem 3.1]{DaviesMandouvalos}}]\label{prop:uniform-estimate-kernel}
    For any $n \geq 2$, there exist two constants $m, M > 0$ depending on $n$ such that:
    \begin{equation} \label{eq:cont04}
        m \, h_n(t, \uprho) \leq K_n(t, \uprho) \leq M \, h_n(t, \uprho), \quad \forall~(t, \uprho) \in (0, +\infty)\times [0, +\infty)\,,
    \end{equation} 
    where
    \begin{equation} \label{eq:cont05} 
        h_n(t, \uprho) = t ^ {-n / 2} \, \mathrm{e}^{-(n - 1)^ 2 t /4 - \uprho^2 / (4 t) - (n - 1) \uprho / 2} \, \left(1 + \uprho + t\right)^{(n - 3) / 2} \, \left(1 + \uprho\right). 
    \end{equation}
\end{prop}

With the help of the proposition above, one can prove that the $\|\cdot\|_{\mathcal{M}}$~norm of the solution of the heat equation \eqref{eq:heat-hyperbolic-source} is controlled by the same norm applied to the inital datum $u_0$ and the source term $f$. This behaviour is outlined in the following proposition which will be useful to derive the convergence of the FDM scheme:

\begin{prop}\label{prop:semigrup-marginire}
For every $T > 0$ there exists a constant $C_{T} > 0$, such that, if $u_0$ and $f$ satisfy Hypothesis~\ref{hyp:biserica-Grivitei}, then the solution $u$ of \eqref{eq:heat-hyperbolic-source} satisfies, for every $t\in [0,T]$ the following estimate:
\[\|u(t)\|_{\mathcal M}\leq C_{T} \left( \|u_0\|_{\mathcal{M}} + \|f\|_{C([0,T],\mathcal{M})} \right).\]
\end{prop}

\section{The first discrete Laplacian}
\label{sec:first-lapl}

\subsection{Finite difference grid}
\label{sec:construction-first-lapl}
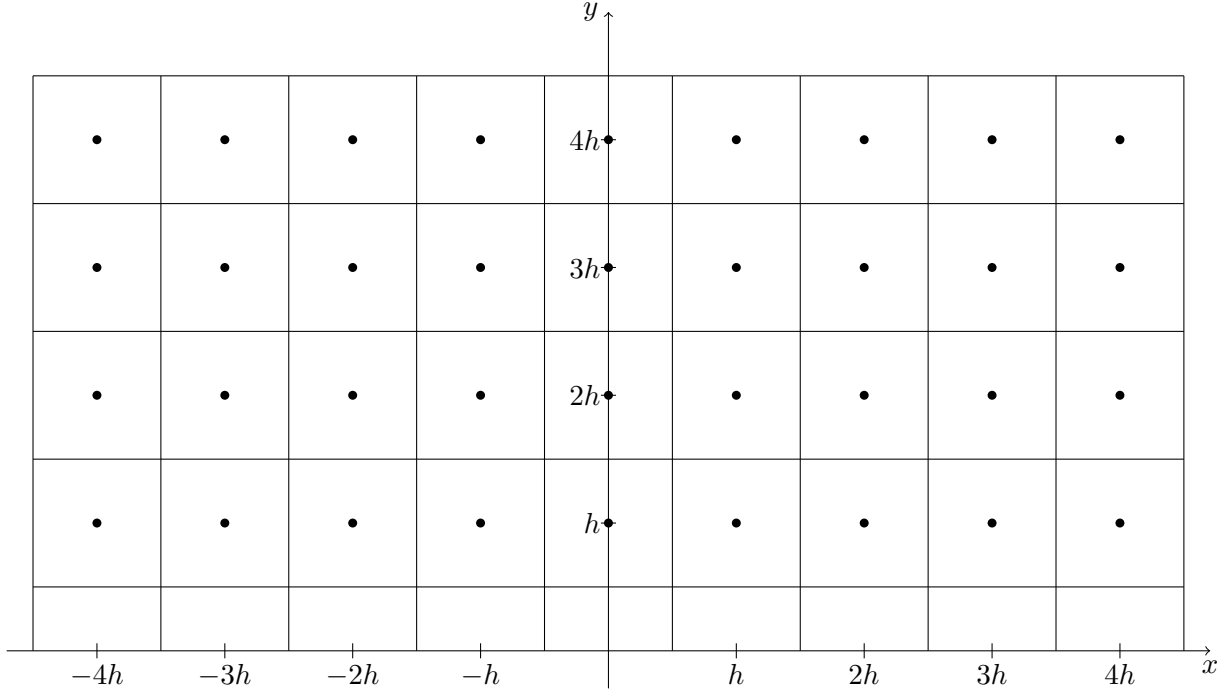
\begin{figure}[h]
\begin{center}
\begin{tikzpicture}
  \draw[->] (-8,0) -- (8,0) node[right] {};
  \draw[->] (0,-0.5) -- (0,8.5) node[above] {};

  \def\sx{1.7}
  \foreach \x in {-5,-4,-3,-2,-1,0,1,2,3,4}
    \draw ({\sx*(\x+0.5)},0) -- ({\sx*(\x+0.5)},{\sx*4.5});

  \node[below] at (8,0) {$x$};
  \node[left] at (0,8.5) {$y$};

  \foreach \x in {-4,-3,-2,-1,0,1,2,3,4}
    \foreach \y in {1,2,3,4}
      \node[circle,fill,inner sep=1.2pt] at ({\sx*\x},{\sx*\y}) {};

  \foreach \x in {-5,-4,-3,-2,-1,0,1,2,3}
    \foreach \y in {0,1,2,3,4}
      \draw ({\sx*(\x+0.5)},{\sx*(\y+0.5)}) -- ({\sx*(\x+1.5)},{\sx*(\y+0.5)});

  \foreach \x in {-4,-3,-2,2,3,4}
  {
    \draw ({\sx*\x},-0.1) -- ({\sx*\x},0.1) node[below=4pt] {$\x h$};
  }
\foreach \x in {-1}
  {
    \draw ({\sx*\x},-0.1) -- ({\sx*\x},0.1) node[below=4pt] {$-h$};
  }  
  \foreach \x in {1}
  {
    \draw ({\sx*\x},-0.1) -- ({\sx*\x},0.1) node[below=4pt] {$h$};
  } 
  \foreach \y in {2,3,4}
  {
    \draw (-0.1,{\sx*\y}) -- (0.1,{\sx*\y}) node[left=2pt] {$\y h$};
  }
  \foreach \y in {1}
  {
    \draw (-0.1,{\sx*\y}) -- (0.1,{\sx*\y}) node[left=2pt] {$h$};
  }

\end{tikzpicture}
\caption{\label{fig:first-grid} The grid corresponding to the first discrete Laplacian on $\mathbb{H}^2$.}
\end{center}
\end{figure}

We shall construct a numerical approximation, $u^h$, to the solution of the problem~\eqref{eq:heat-hyperbolic-source} on an equidistant spatial grid defined on $\HH$ by the following nodes:
\begin{equation} \label{eq:spatial-grid}
    (i h,j h) \in \HH \quad \, i \in \mathbb{Z}, \, j \in \mathbb{N}^\ast\,,
\end{equation}
where $h > 0$ is the uniform stepsize of the FDM grid, see Figure~\ref{fig:first-grid}.

\paragraph{Finite difference cell.} Around each spatial point of the discrete FDM grid, $(ih, jh) \in \HH$, we define the following cells:
\begin{equation*}
    \mathcal{C}_h^{i, j} \coloneqq \left[ih - \frac h 2, ih + \frac h 2\right] \times \left[jh - \frac h 2, jh + \frac h 2\right]\,,
\end{equation*}
with the hyperbolic area of
\begin{equation*}
    \left|\CC_h^{i,j}\right|_g:=\displaystyle \int_{\mathcal{C}_h^{i, j}} \dd{\upmu} = \displaystyle \int_{ih - h / 2}^{ih + h / 2} \int_{jh - h / 2}^{jh + h / 2} \dfrac{1}{x_2^2} \dd{\x} = \dfrac{1}{j - \frac 1 2} - \dfrac{1}{j + \frac 1 2} = \dfrac{1}{j^2 - \frac 1 4}.
\end{equation*}

\paragraph{Space of grid functions.}
In order to represent the value of the approximated solution $u^h$ in each of the cells above, we introduce the following space of grid functions:
\[\ell_h^2\coloneqq \left\{\left(v^h_{i, j}\right)_{i \in \mathbb{Z}, \, j \in \mathbb{N}^\ast}:  \Vert v^h \Vert_{\ell_h^2} \coloneqq \langle v^h, v^h\rangle_{\ell_h^2}^{1/2} < +\infty \right\},\]
where the aforementioned scalar product is defined as:
\begin{equation*}
    \langle u^h, v^h \rangle_{\ell_h^2} \coloneqq \displaystyle \sum_{i \in \mathbb{Z}, \, j \in \mathbb{N}^\ast}  \left|\CC_h^{i,j}\right|_g u^h_{i, j} v^h_{i, j} =  \displaystyle \sum_{i \in \mathbb{Z}, \, j \in \mathbb{N}^\ast}  \dfrac{1}{j^2 - \frac 1 4} u^h_{i, j} v^h_{i, j}\,,
\end{equation*}

\paragraph{Projection on the space of grid functions.} Next, we define the projection of a function $v \in \mathrm{L}^2(\HH)$ on the space of grid functions, $\ell^2_h$ as:
\begin{equation} \label{eq:projection-first-lapl}
    \Pi_h : \mathrm{L}^2(\HH) \longrightarrow \ell_h^2, \qquad (\Pi_h v)_{i, j} \coloneqq \frac{1}{\left|\CC_h^{i,j}\right|_g} \int_{\mathcal{C}_h^{i, j}} v(\x) \dd{\mu(\x)} =\left(j^2 - \dfrac{1}{4}\right) \displaystyle \int_{\mathcal{C}_h^{i, j}} v(x_1, x_2) \dfrac{1}{x_2^2} \dd{\x}.
\end{equation}
We note that the projection operator  $\Pi_h$ given by \eqref{eq:projection-first-lapl} is contractive, meaning that $\Vert \Pi_h(v) \Vert_h\leq \|v\|_{\mathrm{L}^2(\HH)}$.

\subsection{Discrete Laplace operator and accuracy}\label{section:First-Discrete-Laplacian}

We aim to construct the discrete FDM counterpart of $\Delta_g$ as a weighted combination of a five-point stencil in the grid in Figure \ref{fig:first-grid}. Namely, we seek to calibrate the weights $\nord, \sud,\vest, \est$ and $\punct$ corresponding to the north, south, west and east direction, as well as the weight of the point itself in the following expression:
\begin{equation} \label{eq:abstract-weights}
    \left(\Delta^{(1)}_h v^h\right)_{i, j} \coloneqq \mathcal{E} v^h_{i+1, j} + \mathcal{W}v^h_{i-1, j} + \mathcal{N} v^h_{i, j+1} + \mathcal{S} v^h_{i, j-1} -  \mathcal{P} v^h_{i, j}\, \quad \forall~i\in\mathbb{Z}, \; j\in\mathbb{N}^\ast\,,
\end{equation}
in order to attain accuracy of order $\mathcal{O}(h^2)$. More precisely, we aim to obtain the following estimate:
\begin{equation}\label{eq:error-vanishes}
    Err_h(v)\coloneqq\|\Delta^{(1)}_h \Pi_h v - \Pi_h \Delta_g v\|_{\ell_h^2}\xrightarrow{h\to 0}0,\, \forall v\in \mathcal{M}.
\end{equation}

The definitions of the $\ell_h^2$ norm and of the projection $\Pi_h$, together with a change of variables leads to:
\begin{equation}\label{eq:lapl1-err-NSEV}
\begin{aligned}
Err_{h}(v) &= \displaystyle  \sum_{i \in \Z, j \geq 2} \frac{1}{\left(j^2 - \frac{1}{4}\right)} \left[\est \left(j^2-\frac{1}{4}\right) \int_{\CC^{i,j}_h} \frac{v(x_1+ h,x_2)}{x_2^2} \dd{\x} +
\vest \left(j^2-\frac{1}{4}\right) \int_{\CC^{i,j}_h} \frac{v(x_1- h,x_2)}{x_2^2} \dd{\x} \right.\\[4pt]
 &\phantom{Err_h(v)} \quad +\nord\left((j+ 1)^2 -\frac{1}{4}\right) \int_{\CC^{i,j}_h} \frac{v(x_1,x_2+ h)}{(x_2+ h)^2}\dd{\x} + \sud \left((j- 1)^2 -\frac{1}{4}\right) \int_{\CC^{i,j}_h} \frac{v(x_1,x_2- h)}{(x_2- h)^2}\dd{\x} \\[4pt]
        &\phantom{Err_h(v)} \quad \left. -  \punct \left(j^2 - \frac{1}{4}\right) \int_{C_h^{i, j}}\dfrac{v(x_1,x_2)}{x_2^2} \dd{\mathbf{x}} - \left(j^2 - \frac{1}{4}\right) \int_{C_h^{i, j}} \left(\partial^2_{x_1} v(x_1, x_2) + \partial^2_{x_2} v(x_1,x_2)\right) \dd{\x} \right]^2\\
        &\phantom{Err_h(v)}\quad+\frac{3}{4} \sum_{i\in \Z} \left[(\Pi_h v)_{i \pm 1, 1} + (\Pi_h v)_{i, 2} - 4 (\Pi_h v)_{i, 1} - \frac{4}{3} \left(\Pi_h \Delta_g v\right)_{i, 1} \right]^2.
\end{aligned}
\end{equation}
For functions $v\in \mathcal{M}$, the last line vanishes at any polynomial rate by Proposition \ref{prop:tail-control}, so we focus on the remaining terms of \eqref{eq:lapl1-err-NSEV}. Their form  suggests using the following Taylor expansions with integral reminder on the horizontal direction:
\[\begin{aligned}
v(x_1\pm h,x_2) - 2v(x_1,x_2) = h^2 \partial_{x_1}^2 v(x_1,x_2) + \frac{h^4}{6}\int_{-1}^1 \partial_{x_1}^4 v(x_1+h\upsigma)(1-|\upsigma|)^3\dd\upsigma.
\end{aligned}\]
Therefore, 
\begin{equation}\label{eq:croytoru-1}
\begin{aligned}
\left(j^2-\frac 1 4\right)\left[v(x_1\pm h,x_2) - 2v(x_1,x_2)\right] &= \left[\left(j^2-\frac 1 4\right)h^2-x_2^2\right] \partial_{x_1}^2 v(x_1,x_2) + x_2^2\partial_{x_1}^2 v(x_1,x_2)\\ & \quad+ \underbrace{\left(j^2-\frac 1 4\right)\frac{h^4}{6}\int_{-1}^1 \partial_{x_1}^4 v(x_1+h\upsigma)(1-|\upsigma|)^3\dd\upsigma}_{\mathcal{O}(h^2)},
\end{aligned}
\end{equation}
where the estimate $\mathcal{O}(h^2)$ follows from the construction of the cell $\mathcal{C}_h^{i,j}$. By the same argument, the term $\left[\left(j^2-\frac 1 4\right)h^2-x_2^2\right]$ is of order $\mathcal{O}(h)$

On the vertical direction we perform the same reasoning, but for the function $x_2\to \dfrac{v(x_1,x_2)}{x_2^2}$:
\begin{equation}\label{eq:croytoru-2}
\begin{aligned}
&\left[(j+1)^2-\frac{1}{4}\right]\frac{u(x_1,x_2+h)}{(x_2+h)^2} + \left[(j-1)^2-\frac{1}{4}\right]\frac{u(x_1,x_2-h)}{(x_2-h)^2}- 2 \left(j^2-\frac{1}{4}\right) \frac{u(x_1,x_2)}{x_2^2}\\
&\quad =\left[\left(j^2-\frac{1}{4}\right)h^2 -x_2^2\right]\partial^2_{x_2}\left[ \frac{u(x_1,x_2)}{x_2^2}\right] +4(jh-x_2) \partial_{x_2}\left[ \frac{u(x_1,x_2)}{x_2^2}\right]+h^2 \partial^2_{x_2}\left[ \frac{u(x_1,x_2)}{x_2^2}\right]\\
&\quad + \underbrace{x_2^2\partial_{x_2}^2\left[ \frac{u(x_1,x_2)}{x_2^2}\right]+4x_2\partial_{x_2}\left[ \frac{u(x_1,x_2)}{x_2^2}\right]+ 2 \frac{u(x_1,x_2)}{x_2^2}}_{\partial_{x_2}u(x_1,x_2)} +\mathcal{O}(h^2),
\end{aligned}
\end{equation}
where the term $(jh-x_2)$ is also $\mathcal{O}(h)$ due to the construction of the FDM grid.

These Taylor expansions allow us to fit the parameters $\nord,\sud,\est,\vest,\punct$ and obtain the following formula for the finite-difference approximation of the Laplace-Beltrami operator $\Delta_g$  on the FDM grid~\eqref{eq:spatial-grid}:
\begin{equation} \label{eq:discrete-first-laplace}
    \left(\Delta^{(1)}_h v^h\right)_{i, j} \coloneqq \left( j^2 - \dfrac{1}{4} \right) \left(v^h_{i+1, j} + v^h_{i-1, j} + v^h_{i, j+1} + v^h_{i, j-1} - 4v^h_{i, j}\right)\, \quad \forall~i\in\mathbb{Z}, \; j\in\mathbb{N}^\ast\,,
\end{equation}
where we employ the convention $v^h_{i, 0} = 0, \; \forall~i\in\mathbb{Z}$. We note that, even though the weight  $\left(j^2-\frac 1 4\right)$ in \eqref{eq:discrete-first-laplace} does not directly depend on the grid parameter $h$, it automatically increases as the grid gets refined, i.e. when $h$ decreases.

Since the terms $\left[\frac{\left(j^2-1/4\right)h^2}{x_2^2}-1\right]$ and $(jh-x_2)$ have Euclidean zero mean on the cell $C_h^{i,j}$, one can improve, by \eqref{eq:lapl1-err-NSEV}-\eqref{eq:croytoru-2}, the order of accuracy of $\Delta_h^{(1)}$ to $\mathcal{O}(h^2)$:

\begin{thm}[Accuracy for the first discrete Laplacian]\label{thm:consistency-first-lapl}
There exists a universal constant $C>0$ such that, for every $v\in\mathcal{M}$ and $h\in (0,1)$ the following estimate holds:
\begin{equation}\label{eq:consistency-estimate-first-lapl}
Err_h(v)\coloneqq \|\Delta^{(1)}_h \Pi_h v - \Pi_h \Delta_g v\|_{\ell_h^2}\leq  C \left[h^2 \|v\|_{\mathcal{M}_\gamma} + \|v(x)\|_{{\rm L}^2(\HH\setminus\mathbb{H}^2_{\frac{2}{5h}})}+\|\Delta_g v(x)\|_{{\rm L}^2(\HH\setminus\mathbb{H}^2_{\frac{2}{5h}})} \right].
\end{equation}
\end{thm}

\subsection{Reduction to a bounded domain}

Further, motivated by the fact that the operator $\Delta^{(1)}_h$ is not dissipative on the entire grid, we introduce the both theoretical and numerically suited 
bounded domain $\HHD$ defined in \eqref{Intro:eq:cine-PLM-i-HD}, for a fixed $D > 0$ which will be later chosen as in~\eqref{Intro:eq:cine-PLM-i-D}. Taking into account the decay of the solution of the heat equation for large $\x$ (refer to Propositions \ref{prop:tail-control} and \ref{prop:semigrup-marginire}), we approximate the solution of~\eqref{eq:heat-hyperbolic-source} with the solution of the discrete version of the following homogeneous Dirichlet initial-boundary value problem:
\begin{eqnarray} \label{eq:Dirichlet-heat}
\left\{
\begin{array}{ll}
\partial_t u(t, \mathbf{x}) = \Delta_g u(t,\mathbf{x})+f(t,\x)\,,
& t \in (0, T],\; \mathbf{x} \in \HHD\,, \\[4pt]
\phantom{\partial_t} u(t, \mathbf{x}) = 0\,,
& t \in [0, T],\; \mathbf{x} \in \partial \HHD\,,\\[4pt]
\phantom{\partial_t} u(0, \mathbf{x}) = u_0(\mathbf{x})\,,
& \mathbf{x} \in \HHD\,.
\end{array}
\right.
\end{eqnarray}
Accordingly, we denote the space of grid functions restricted to this domain by
\[
\ell_{h,D}^2\coloneqq\left\{(v^h_{i,j})_{i\in Z_1, j\in Z_2} \in \ell_h^2, \rm{where}~Z_1 = \overline{- N, N}\,, Z_2 = \overline{1, M}\right\}\,,
\]
where $N = \left\lfloor \frac{D}{h} \right\rfloor$ and $M = \left\lfloor \frac{1}{h}\left(D - \frac{1}{D}\right) \right\rfloor + 1$.
As $\ell_{h,D}^2 \subset \ell_h^2$, we equip it with the same norm, whilst, for a function $v\in \mathrm{L}^2(\HH)$, we define the projection on the space of grid functions $\ell^2_{h,D}$ as the corestriction of~\eqref{eq:projection-first-lapl}:
\[
    \Pi_{h, D} : \mathrm{L}^2(\HH) \longrightarrow \ell_{h,D}^2, \qquad (\Pi_{h, D} v)_{i, j} \coloneqq (\Pi_{h} v)_{i, j}.
\]
Finally, we define the discrete Laplacian restricted to $\HHD$ as
\[
    \Delta^{(1)}_{h, D}v^h \coloneqq \Delta^{(1)}_h v^h\,,
\]
together with the convention that $v^h_{i,j}=0$ for $(i,j)\notin Z_1\times Z_2$.

\begin{rem}\label{rem:consistency-first-laplace-sos}
Choosing $D\coloneqq D_{h,\gamma,\zeta}$ as in~\eqref{Intro:eq:cine-PLM-i-D} and applying Propositions \ref{prop:tail-control} and \ref{prop:semigrup-marginire} to $u$, the solution of~\eqref{eq:heat-hyperbolic-source} with $u_0$ and $f$ satisfying Hypothesis~\ref{hyp:biserica-Grivitei}, we obtain that:
\[\|\Delta^{(1)}_{h,D} \Pi_{h,D} u(t) - \Pi_{h,D} \Delta_g u(t)\|_{\ell_{h,D}^2} \leq h^2 C_{\gamma,\zeta,T} \left( \|u_0\|_{\mathcal{M}} + \|f\|_{C([0,T],\mathcal{M})} \right).\]
\end{rem}

\subsection{Poincar\' e inequality}
In this section, we will show that the numerical scheme corresponding to the aforementioned discrete Laplace operator preserves the asymptotic decay of the continuous homogeneous heat equation on $\HH$. More precisely, we will show that the Poincar\'e inequality corresponding to this discrete Laplacian has the same optimal constant as its continuous counterpart in Section \ref{sec:poincare-continuous}:
\begin{thm}[Poincar\' e inequality for the first discrete Laplacian]
\label{thm:first-laplacian-poincare}
Assume that $v \coloneqq v^h \in \ell_h^2$ has compact support. Then, the following inequality holds true:
\begin{equation}\label{eq:poincare-first-lapl}
-\langle \Delta^{(1)}_h v, v \rangle_{\ell_h^2} \geq \frac{1}{4} \|v\|_{\ell_h^2}^2. 
\end{equation}
Moreover, the constant $\frac{1}{4}$ is sharp, meaning that it cannot be improved.
\end{thm}
In the proof of Theorem \ref{thm:first-laplacian-poincare} we use the following lemma, which provides an integration by parts formula for the discrete Laplace operator $\Delta_h^{(1)}$:
\begin{prop}
\label{prop:first-laplacian-dissipative}
    For any compactly supported sequence $v^h\in \ell_h^2$ (i.e., only a finite number of elements are not null), the following equality holds true:
    \begin{equation} \label{eq:integration_parts}
        \langle \Delta^{(1)}_h v^h, v^h\rangle = - \displaystyle \sum_{i\in\mathbb{Z}, j\in\mathbb{N}} \left[\left(v^h_{i+1, j} - v^h_{i, j} \right)^2 + \left(v^h_{i, j+1} - v^h_{i, j} \right)^2 \right]\,,
    \end{equation}
with the convention $v^h_{i, 0} = 0, \; \forall~i\in\mathbb{Z}$.
\end{prop}
\begin{proof}[{Proof of Theorem \ref{thm:first-laplacian-poincare}}]
    According to Proposition \ref{prop:first-laplacian-dissipative}, we need to prove the following inequality:
\begin{equation}
    \label{eq:poincare-first-explicit}
    \sum_{i\in \Z, j\in \NN} \left[(v_{i+1,j}-v_{i,j})^2 + (v_{i,j+1}-v_{i,j})^2 \right]\geq \frac{1}{4} \sum_{i\in \Z, j\in \NN} \frac{1}{j^2-1/4} |v_{i,j}|^2,
\end{equation}
with the convention $v_{i,0}=0$, $\forall i\in \Z$.
Since the weights involved in the right-hand side term above only depend on the second variable (i.e. $j\in \NN$), we consider first a single-variable sequence $(w_j)_{j\in \NN^\ast}$ and claim that:
\begin{equation}
    \label{eq:poincare-first-unidimnesional}
     \sum_{j\in \NN}  (w_{j+1}-w_{j})^2 \geq \frac{1}{4} \sum_{j\in \NN} \frac{1}{j^2-1/4} |w_{j}|^2,
\end{equation}
of course, employing the convention $w_0=0$.
The proof of this claim is inspired by the study of a discrete Hardy inequality on the line \cite{david2022DiscreteHardyLine}. Indeed, taking into account that $\sum_{j\geq0} w_{j}^2 =\sum_{j\geq0} w_{j+1}^2$, we write the left-hand side term above as:
\begin{equation}
    \label{eq:poincare-first-unidimnesional-e1}
     \sum_{j\in \NN}  (w_{j+1}-w_{j})^2 = \sum_{j\in \NN^\ast} \left(2-\sqrt{\frac{j-1}{j}}-\sqrt{\frac{j+1}{j}}\right) |w_j|^2 +  \sum_{j\in \NN^\ast} \left(\sqrt[4]{\frac{j+1}{j}}w_j-\sqrt[4]{\frac{j}{j+1}}w_{j+1}\right)^2.
\end{equation}
Now, \eqref{eq:poincare-first-unidimnesional} follows easily since, by direct computation, we have that for every $j\geq 1$,
\begin{equation*}
    2-\sqrt{\frac{j-1}{j}}-\sqrt{\frac{j+1}{j}} \geq \frac{1}{4(j^2-1/4)}.
\end{equation*}
To obtain the desired inequality \eqref{eq:poincare-first-explicit}, all we need to do is to sum up the instances of \eqref{eq:poincare-first-unidimnesional} for every sequence $(v_{i,j})_{j\in \NN^\ast}$ as $i\in \Z$.

To prove the sharpness of the constant $\frac{1}{4}$, first we will construct a minimizing sequence for the one-dimensional quantity:
\begin{equation*}
\mathcal{I}(w):=\frac{\displaystyle\sum_{j\in \NN} (w_{j+1}-w_{j})^2}{\displaystyle\sum_{j\in \NN} \frac{1}{j^2-1/4} |w_{j}|^2}.
\end{equation*}
Since a minimizing sequence should annihilate the last term of \eqref{eq:poincare-first-unidimnesional-e1}, we will consider $w_j=\sqrt{j}$, but at the same time we properly cut it off in order to become compactly supported. More precisely, for every $n\in \NN^\ast$, $n\geq 2$, we define (see \cite[Section 2.2]{david2022DiscreteHardyLine}):
\begin{equation*}
w^n_j\coloneqq\sqrt{j}\cdot\left\{\begin{array}{cc}
    1,& 1\leq j<n;\\
    \frac{2\log n - \log j}{\log n}, & n\leq j\leq n^2;\\
    0, & j>n^2.
\end{array}\right.
\end{equation*}

From \eqref{eq:poincare-first-unidimnesional} and \eqref{eq:poincare-first-unidimnesional-e1} we obtain that: 
\begin{equation*}
0\leq I(w^n)-\frac{1}{4}=\frac{\displaystyle\sum_{j\in \NN^\ast} \left(\sqrt[4]{\frac{j+1}{j}}w^n_j-\sqrt[4]{\frac{j}{j+1}}w^n_{j+1}\right)^2+\sum_{j\in \NN^\ast}\left(2-\sqrt{\frac{j-1}{j}}-\sqrt{\frac{j+1}{j}} -\frac{1}{4j^2-1}\right) |w_j^n|^2}{\displaystyle\sum_{j\in \NN} \frac{1}{j^2-1/4} |w_{j}^n|^2}.
\end{equation*}

Computing all the terms explicitly and using standard properties of the generalized harmonic series,
we deduce that:
\begin{equation}\label{eq:poincare-first-unidim-estimate-5}
0\leq \mathcal{I}(w^n)-\frac{1}{4}\lesssim \frac{\frac{1}{\log n}+\frac{\pi^2}{6}}{\log n} \lesssim \frac{1}{\log n}\xrightarrow{n\to \infty} 0,
\end{equation}
which means that $(w^n)_{n\geq 2}$ is a minimizing sequence for the functional $I$, thus proving the sharpness of the constant $\frac{1}{4}$ for the one-dimensional Poincar\' e-type inequality \eqref{eq:poincare-first-unidimnesional}.

Based on the one-dimensional minimising sequence $(w^n)_{n\geq 2}$, we now construct a minimizing double sequence $(v^{m,n})_{m,n\geq 2}$ for the functional:
\begin{equation*}
\mathcal{J}(v):=\frac{\displaystyle\sum_{i\in \Z, j\in \NN} \left[(v_{i+1,j}-v_{i,j})^2 + (v_{i,j+1}-v_{i,j})^2 \right]}{\displaystyle\sum_{i\in \Z, j\in \NN} \frac{1}{j^2-1/4} v_{i,j}^2},
\end{equation*}
which, in turn, will prove the optimality of the constant $\frac{1}{4}$ in \eqref{eq:poincare-first-explicit}. Indeed, let us define $(v^{m,n}_{i,j})_{i\in \Z,j\in \NN^\ast}$ in the following way:
\[v^{m,n}_{i,j}\coloneqq\left\{\begin{array}{cc}
w_j^n \left(1-\frac{|i|}{m}\right), & |i|\leq m\\
0, & \text{otherwise}
\end{array}\right.
\]
Direct computations, together with \eqref{eq:poincare-first-unidim-estimate-5}, lead to:
\[
\begin{aligned}
0\leq \mathcal{J}(v^{m,n})-\frac{1}{4} &\lesssim \frac{1}{\log n} + \frac{\displaystyle \sum_{j\in \NN} |w^n_j|^2}{\displaystyle \sum_{j\in \NN} \frac{1}{j^2-1/4} |w^n_j|^2 } \cdot \frac{6}{2m^2+1}.
\end{aligned}
\]
In the conclusion, for a fixed $\varepsilon>0$, we choose $n$ small enough such that $\frac{1}{\log n}\leq \frac{\varepsilon}{2}$ and then pass $m$ to infinity to get that the constant $\frac{1}{4}$ is sharp for \eqref{eq:poincare-first-explicit}.
\end{proof}

\section{The second discrete Laplacian}
\label{sec:second-lapl}
\subsection{Finite difference grid}
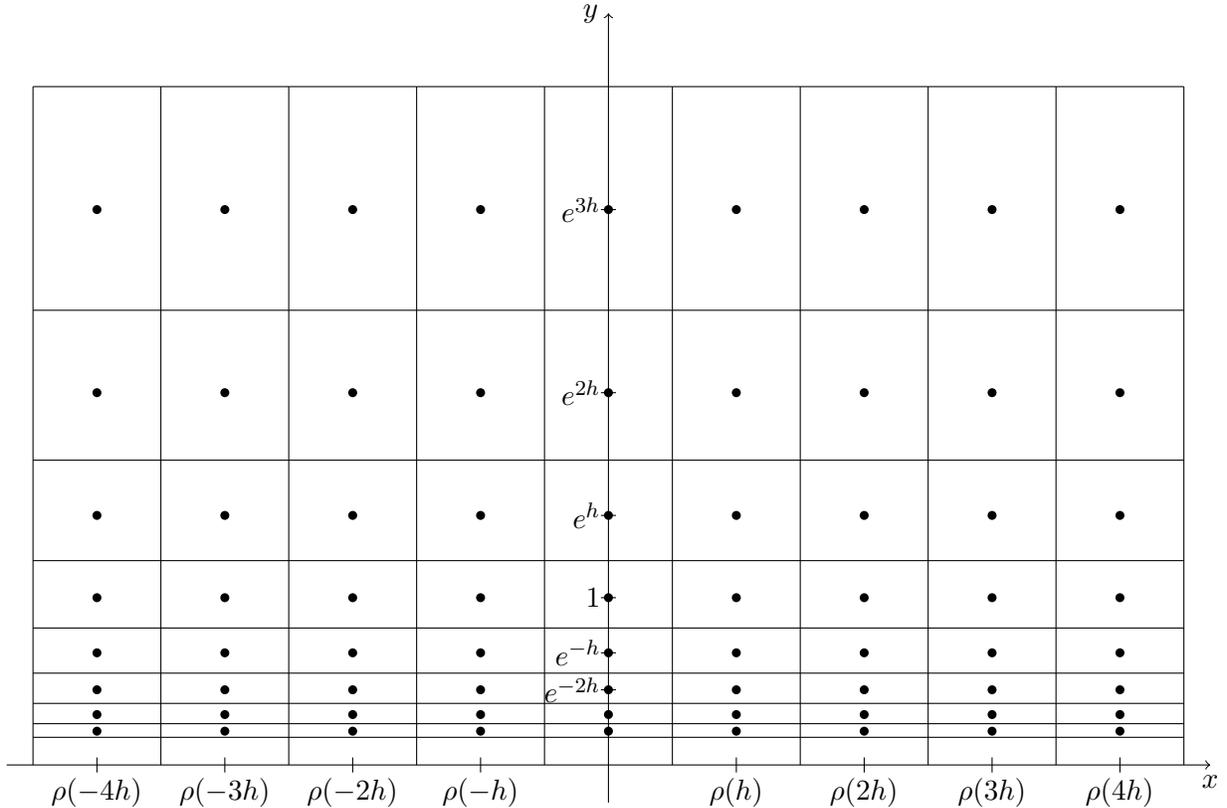
\begin{figure}[h]
\begin{center}
\begin{tikzpicture}
  \draw[->] (-8,0) -- (8,0) node[right] {};
  \draw[->] (0,-0.5) -- (0,10) node[above] {};

\def\ra{2.5}
\def\sx{1.7}
  \foreach \x in {-5,-4,-3,-2,-1,0,1,2,3,4}
    \draw ({\sx*(\x+0.5)},0) -- ({\sx*(\x+0.5)},{exp((5+0.5)/\ra)});

  \node[below] at (8,0) {$x$};
  \node[left] at (0,10) {$y$};

  \foreach \x in {-4,-3,-2,-1,0,1,2,3,4}
    \foreach \y in {-2,...,5}
      \node[circle,fill,inner sep=1.2pt] at ({\sx*\x},{exp(\y/\ra)}) {};

  \foreach \x in {-5,-4,-3,-2,-1,0,1,2,3}
    \foreach \y in {-3,...,5}
      \draw ({\sx*(\x+0.5)},{exp((\y+0.5)/\ra)}) -- ({\sx*(\x+1.5)},{exp((\y+0.5)/\ra)});

  \foreach \x in {-4,-3,-2,2,3,4}
  {
    \draw ({\sx*\x},-0.1) -- ({\sx*\x},0.1) node[below=4pt] {$\x\rho(h)$};
  }
  \foreach \x in {-1}
  {
    \draw ({\sx*\x},-0.1) -- ({\sx*\x},0.1) node[below=4pt] {$-\rho(h)$};
  }
   \foreach \x in {1}
  {
    \draw ({\sx*\x},-0.1) -- ({\sx*\x},0.1) node[below=4pt] {$\rho(h)$};
  }
  \foreach \y in {-2,2,3}
  {
    \draw (-0.1,{exp((\y+2)/\ra)}) -- (0.1,{exp((\y+2)/\ra)}) node[left=2pt] {$e^{\y h}$};
  }
    \foreach \y in {-1}
  {
    \draw (-0.1,{exp((\y+2)/\ra)}) -- (0.1,{exp((\y+2)/\ra)}) node[left=2pt] {$e^{-h}$};
  }
 \foreach \y in {0}
  {
    \draw (-0.1,{exp((\y+2)/\ra)}) -- (0.1,{exp((\y+2)/\ra)}) node[left=2pt] {$1$};
  }
   \foreach \y in {1}
  {
    \draw (-0.1,{exp((\y+2)/\ra)}) -- (0.1,{exp((\y+2)/\ra)}) node[left=2pt] {$e^{h}$};
  }

\end{tikzpicture}
\caption{\label{fig:second-grid} The grid corresponding to the second discrete Laplacian on $\mathbb{H}^2$.  $\rho(h)=2\sinh\left(\frac{h}{2}\right).$}
\end{center}
\end{figure}

The design of the second variant of discrete Laplacian on $\HH$ aims to balance between alignment to the hyperbolic geometry and the computational simplicity of finite differences. The primary feature of the grid is that, along each vertical and horizontal line, the hyperbolic width of the divisions remains constant. However, while the division width is consistent across different vertical lines, the horizontal lines still retain a subtle influence from Euclidean geometry: the grid points on the same horizontal line become hyperbolically far apart as the horizontal line approaches the baseline. More precisely, this grid, depicted in Figure \ref{fig:second-grid}, is defined by the nodes:
\begin{equation} \label{eq:spatial-grid-2}
    (i\rho(h),e^{jh}) \in \HH\,, \quad \, i, j \in \mathbb{Z}\,,
\end{equation}
where
\begin{equation}
    \label{defRho}
    \rho(h)=2\sinh\left(\frac{h}{2}\right)\,, \quad h>0\,.
\end{equation}

The particular choice of $\rho(h)$ in \eqref{defRho} is justified by the fact that, on the horizontal line $x_2=1$, the hyperbolic distances between consecutive points is exactly $h$. We also remark, that, as $h$ approaches zero, $\rho(h)\sim h$. Furthermore, the hyperbolic distance between any consecutive nodes on a vertical line of this grid is $h$.

\paragraph{Finite difference cell.} The choice of the discretisation \eqref{eq:spatial-grid-2} induces the partition of $\HH$ into cells, similar to the ones in Section \ref{sec:construction-first-lapl}:
\begin{equation}
    \label{eq:cell-2}
    \CC_{h}^{i,j}\coloneqq \left[\left(i-\frac 1 2\right)\rho(h),\left(i+\frac 1 2 \right)\rho(h)\right]\times \left[e^{jh-\frac h 2},e^{jh+\frac h 2}\right].
\end{equation}
The hyperbolic area of this cell is then:
\begin{equation}\label{eq:cell-area-lapl-2}
\left|\CC_{h}^{i,j}\right|_{g}:=\int_{i\rho(h)-\rho(h)/2}^{i\rho(h)+\rho(h)/2}\int_{e^{jh-h/2}}^{e^{jh+h/2}}\frac{1}{x_2^2} \dd{\x}=(\rho(h))^2e^{-jh}.
\end{equation}

\paragraph{Scalar product, norm and projection on grid}
Similar to Section \ref{sec:construction-first-lapl}, this newly defined grid induces a scalar product, a Hilbert space norm and a projection on the space of grid functions:
\[\ell_h^2 \coloneqq \left\{ \left(v^h_{i, j}\right)_{i,j \in \mathbb{Z}} : \|v^h\|_{\ell_h^2}^2 \coloneqq\langle v^h,v^h\rangle_{\ell_h^2} < +\infty\right\},\]
where the scalar product is defined as:
\[\langle u^h,v^h \rangle_{\ell_h^2} \coloneqq \sum_{i,j\in \Z} (\rho(h))^2 e^{-jh} u^h_{i,j} v^h_{i,j}.\]

In this setting, the projection operator takes the following form:
\begin{equation}\label{eq:projection-2}  \Pi_h : \mathrm{L}^2(\HH) \longrightarrow \ell_h^2,  \qquad (\Pi_h v)_{i, j} \coloneqq \frac{e^{jh}}{(\rho(h))^2} \displaystyle \int_{\mathcal{C}_h^{i, j}} v(x_1, x_2) \dfrac{1}{x_2^2} \dd{\x}.
\end{equation}
The contractivity of the projection $\Pi_h$ also holds in this case: $\|\Pi_h v\|_{l_h^2} \leq \|v\|_{{\rm L}^2(\HH)}$.

\begin{rem}\label{rem:equidistant-complicated}
We have chosen to adapt this grid only partially to the geometry of the half-plane model of $\HH$ for both theoretical and practical reasons, because, unlike the Euclidean space, the hyperbolic space expands exponentially as the distance to any fixed point increases.  We refer to \cite[Figure 2]{YuDeSa2019} for the behaviour of distances towards the infinity boundary $x_2=0$ of the half-plane model and note that, in an equidistant grid, each node on one row would require two nodes on the row below. Because of that, it is impossible to construct a hyperbolic equidistant grid both axes' direction and to index it by $\Z\times \Z$, such that the neighbours have consecutive indexes.
\end{rem}

\subsection{Discrete Laplace operator and accuracy}\label{sec:construction-second-lapl} 
Based on the grid illustrated in Figure\eqref{fig:second-grid}, we develop our second discrete FDM Laplace operator in the spirit of \eqref{eq:abstract-weights}. Once again, our aim is to obtain an accuracy error $Err_h(v)\coloneqq\|\Delta^{(2)}_h \Pi_h v - \Pi_h \Delta_g v\|_{\ell_h^2}$ that vanishes as $h$ approaches $0$. 

Using the change of variables $x_1 \leftarrow x_1\pm \rho(h)$ and $x_2\leftarrow x_2\, e^{\pm h}$, we obtain:
\begin{equation}\label{eq:2-lapl-NSEV}
\begin{aligned}
    Err_h(v)^2 &= \sum_{i,j\in\Z} \frac{e^{jh}}{(\rho(h))^2} \left\{ \int_{\CC_h^{i,j}} \left[\est\, v(x_1+ \rho(h),x_2)+\vest\, v(x_1- \rho(h),x_2)+\nord\, v(x_1,x_2e^h)+ \sud\, v(x_1,x_2e^{-h}) \right]\frac{1}{x_2^2} \dd{\x}\right.\\
     &\quad \left.  +\int_{\CC_h^{i,j}} \left[\punct\, v(x_1,x_2)\frac{1}{x_2^2}- \left(\partial^2_{x_2} v(x_1,x_2)+\partial^2_{x_2} v(x_1,x_2)\right)\right]\dd{\x}\right\}^2,
\end{aligned}\end{equation}
where the weights $\nord,\sud,\est,\vest$ and $\punct$ play the same role as in \eqref{eq:abstract-weights}. To identify the weights corresponding to the horizontal direction, we emply  a Taylor expansion similar to \eqref{eq:croytoru-1}:
\begin{equation}\label{eq:croytoru-3}
\begin{aligned}
    \frac{e^{2jh}}{x_2^2\, \rho(h)^2}\left[v(x_1\pm \rho(h),x_2) - 2v(x_1,x_2)\right] - \partial_{x_1}^2 v(x_1,x_2) &= \left[\frac{e^{2jh}}{x_2^2}-1 \right]\partial_{x_1}^2 v(x_1,x_2)   +\mathcal{O}(h^2).
\end{aligned}
\end{equation}
The expression above is of order $\mathcal{O}(h)$ by the construction of the grid.

For the vertical direction, we perform two Taylor expansions for the function $\varphi(s)=v(x_1,x_2+sx_2)$ on the intervals $\left(e^{-h}-1,0\right)$ and $\left(0,e^{h}-1\right)$, respectively. A linear combination of those expansions leads to:
\begin{equation}\label{eq:croytoru-4}
    \begin{aligned}
    &\frac{2}{e^h+1} v(x_1,x_2e^h)+\frac{2 e^h}{e^h+1} v(x_1,x_2e^{-h}) - 2 v(x_1,x_2) -  (\rho(h))^2 x_2^2 \partial^2_{x_2} v(x_1,x_2) \\
    &\quad = \underbrace{(\rho(h))^4\left[\frac{1}{3} x_2^3 \partial_{x_2}^3 v(x_1,x_2) + \frac{e^{h/2}}{e^h+1} \int_0^1 \sum_{\pm} e^{\pm 3h/2} x_2^4 \partial^4_{x_2} v(x_1,x_2+\upsigma (e^{\pm h}-1)x_2)\frac{(1-\upsigma)^3}{3} \dd{\upsigma}\right]}_{\mathcal{O}(h^4)}.
    \end{aligned}
\end{equation}

By \eqref{eq:2-lapl-NSEV}-\eqref{eq:croytoru-4}, we deduce the form of the discrete FDM Laplace operator corresponding to this grid: 
\begin{equation} \label{eq:laplacian-2}
(\Delta^{(2)}_h v^h)_{i,j} \coloneqq \frac{1}{(\rho(h))^2}\left[e^{2jh}(v^h_{i+1,j}+v^h_{i-1,j}-2 v^h_{i,j}) + \frac{2}{e^h+1} v^h_{i,j+1} + \frac{2 e^h}{e^h+1} v^h_{i,j-1}- 2 v^h_{i,j}\right].
\end{equation}

\begin{rem}\label{rem:second-laplace-greutati}
Regarding the choice of the weights of this discrete Laplace operator, we note that:
\begin{enumerate}[1)]
\item Both the horizontal and vertical weights of the discrete Laplace operator resemble the curvature-driven behaviour of diffusion in that particular cell.
\item The weights corresponding to the vertical direction of the half-space model do not depend on the particular place in space where the operator is applied (i.e., do not depend on $i$ and $j$), thus emphasising a universal capture of the curvature of $\HH$ on this particular direction. Moreover, the space invariance of the weights implies that the part of \eqref{eq:laplacian-2} corresponding to the vertical direction is a dissipative operator on the whole grid, i.e., on $\ell_h^2$.
\end{enumerate}
\end{rem}

Since the term $\left[\frac{e^{2jh}}{x_2^2}-1 \right]$ has Euclidean zero mean on the cell $C_h^{i,j}$, one can improve using \eqref{eq:2-lapl-NSEV}-\eqref{eq:croytoru-3} the order of accuracy of $\Delta_h^{(2)}$ to $\mathcal{O}(h^2)$:

\begin{thm}[Accuracy for the second discrete Laplacian]
\label{thm:accuracy-2-lapl}
There exists a universal constant $C>0$ such that for every $v\in \mathcal{M}$ and $h\in (0,\frac 1 2)$, the following estimate holds:
\begin{equation}
\label{eq:consistency-estimate-second-lapl}
Err_h(v)\coloneqq \|\Delta^{(2)}_h \Pi_h v - \Pi_h \Delta_g v\|_{\ell_h^2}\leq  h^2\,\|v\|_{\mathcal{M}},
\end{equation}
\end{thm}

However, the operator \eqref{eq:laplacian-2} is not dissipative on the whole grid and thus, similar to the construction in Section~\ref{sec:construction-first-lapl}, we denote the space of grid functions corresponding to $\Delta^{(2)}_h$ restricted to $\HHD$  by
\[
\ell_{h,D}^2\coloneqq\left\{(v^h_{i,j})_{i\in Z_1, j\in Z_2} \in \ell_h^2, \rm{with}~Z_1 = \overline{- N, N}\,, Z_2 = \overline{-M, M}\right\}\,,
\]
where $N = \left\lfloor \frac{D}{\rho(h)} \right\rfloor$ and $M =  \left\lfloor \frac{\mathrm{log}(D)}{h} \right\rfloor$. We remark that the presence of the logarithm in the vertical direction leads to a significant reduction of the number of nodes employed in the numerical computation and thus optimising both computation speed and memory. For more details, see Tables~\ref{table01} and~\ref{table02}.

\begin{rem}\label{rem:consistency-second-laplace-sos}
Similarly to Remark~\ref{rem:consistency-first-laplace-sos}, we obtain the desired consistency estimate:
\[\|\Delta^{(2)}_{h,D} \Pi_{h,D} u(t) - \Pi_{h,D} \Delta_g u(t)\|_{\ell_{h,D}^2} \leq h^2 C_{\gamma,\zeta,T} \left( \|u_0\|_{\mathcal{M}} + \|f\|_{C([0,T],\mathcal{M})} \right),\]
where $D \coloneqq D_{h, \gamma, \zeta}$ is chosen according to~\eqref{Intro:eq:cine-PLM-i-D}.
\end{rem}

\subsection{Poincar\' e inequality}
In the sequel, we will prove a Poincar\' e-like inequality for the discrete Laplace operator \eqref{eq:laplacian-2}, with the optimal constant converging to $\frac{1}{4}$ (i.e. the constant in the continuous setting) as $h$ approaches zero:

\begin{thm}[Poincar\' e inequality for the second discrete Laplacian]\label{thm:lapl-2-poincare}
Let $v \coloneqq v^h \in \ell_h^2$ be compactly supported. Then, the following inequality holds true:
\begin{equation}
    - \langle \Delta^{(2)}_h v, v \rangle_{\ell_h^2} \geq C_h \|v\|_{\ell_h^2}^2,
\end{equation}
where the constant $C_h \coloneqq \frac{2 e^h}{(e^h+1)(1+e^{h/2})^2} $ is optimal and satisfies:
\[\lim_{h\to 0} C_h =\frac{1}{4}.\]
\end{thm}
In the proof of Theorem \ref{thm:lapl-2-poincare}, we make use of the following summation by parts formula:
\begin{lem}
\label{eq:lapl-2-dissipative}
    Let $v^h \in \ell_h^2$ be a compactly supported double sequence (i.e. $v^h_{i.j}$ vanishes except for a finite number of indices $(i,j)\in \Z\times \Z$). Then, the following identity holds true:
    \[\langle \Delta^{(2)}_h v^h, v^h \rangle_{\ell_h^2} = - \sum_{i,j\in \Z} e^{jh} (v^h_{i+1,j}-v^h_{i,j})^2 -\frac{2}{e^h+1} \sum_{i,j\in \Z} \frac{1}{e^{jh}} (v^h_{i,j+1}-v^h_{i,j})^2.\]
\end{lem}

\begin{proof}[{Proof of Theorem \ref{thm:lapl-2-poincare}}]
First, we write the inequality that we aim to prove:
\begin{equation}
    \label{eq:lapl-2-poincare-explicit}
    \sum_{i,j\in \Z} e^{jh} (v_{i+1,j}-v_{i,j})^2 +\frac{2}{e^h+1} \sum_{i,j\in \Z} \frac{1}{e^{jh}} (v_{i,j+1}-v_{i,j})^2 \geq C_h \, (\rho(h))^2 \sum_{i,j\in \Z} \frac{1}{e^{jh}} |v_{i,j}|^2.
\end{equation}
As in the case of the first Laplacian (Theorem \ref{thm:first-laplacian-poincare}), we prove a Poincar\' e-type inequality for an one-dimensional sequence $(w_j)_{j\in \Z}$:
\begin{equation}\label{eq:lapl-2-poincare-unidim-var}
\frac{2}{e^h+1} \sum_{j\in \Z } \frac{1}{e^{jh}} (w_{j+1}-w_j)^2 \geq C_h\, (\rho(h))^2 \sum_{j\in \Z} \frac{1}{e^{jh}} |w_j|^2,
\end{equation}
which, by direct computation, is equivalent to:
\begin{equation}\label{eq:lapl-2-poincare-unidim}
\sum_{j\in \Z} \frac{1}{e^{jh}} (w_{j+1}-w_j)^2 \geq (e^{h/2}-1)^2 \sum_{j\in \Z} \frac{1}{e^{jh}} |w_j|^2,
\end{equation}
Indeed, this inequality is provd by noticing that the following identity holds:
\begin{equation}
    \label{eq:lapl-2-poincare-unidim-e1}
    \sum_{j\in \Z} \frac{1}{e^{jh}} (w_{j+1}-w_j)^2 = (e^{h/2}-1)^2 \sum_{j\in \Z} \frac{1}{e^{jh}} w_j^2 + \sum_{j\in \Z} \left(\frac{1}{e^{jh/2-h/4}} w_j - \frac{1}{e^{jh/2+h/4}} w_{j+1} \right)^2.
\end{equation}
Eventually, to finish the proof of inequality \eqref{eq:lapl-2-poincare-explicit}, we sum the instances of \eqref{eq:lapl-2-poincare-unidim-var} for $w_j=v_{i,j}$ over $i\in \Z$.

In the sequel, we prove the optimality of the constant $C_h$ for the Poincar\' e inequality \eqref{eq:lapl-2-poincare-explicit}. In this sense, we construct a family of compactly supported double sequences $(v^{m,n}_{i,j})_{i,j\in \Z}$ indexed over the integer parameters $m,n\geq 2$:
\begin{equation}
\label{eq:extremisers-poincare-2}
v^{m,n}_{i,j} = \left\{
\begin{array}{cc}
\left(1-\frac{|i|}{m}\right)e^{jh/2}, & |i|\leq m\,, |j|\leq n;\\
0, & \text{otherwise}.
\end{array}
\right.
\end{equation}
We note that the part depending on $j$ (i.e. $e^{jh/2}$) was chosen as to annihilate the terms of the last sum in \eqref{eq:lapl-2-poincare-unidim-e1}. We finish our proof by showing that the sequence constructed in \eqref{eq:extremisers-poincare-2} is a minimizing sequence for:
\[I(v)=\frac{\displaystyle\sum_{i,j\in \Z} e^{jh} (v_{i+1,j}-v_{i,j})^2 +\frac{2}{e^h+1} \sum_{i,j\in \Z} \frac{1}{e^{jh}} (v_{i,j+1}-v_{i,j})^2} {\displaystyle \sum_{i,j\in \Z} \frac{1}{e^{jh}} v_{i,j}^2}
.\]
Indeed, by direct computation,
\[
\begin{aligned}
I(v^{m,n}) &= \frac{\displaystyle \sum_{i= -m}^{m-1} \frac{1}{m^2} \sum_{|j|\leq n} e^{2jh} + \frac{2}{e^h+1} \sum_{|i|\leq m} \left(1-\frac{|i|}{m}\right)^2 \left[ \sum_{j=-n}^{n-1} \frac{1}{e^{jh}}e^{jh}\left(e^{h/2}-1\right)^2 + 1+\frac{e^{-nh}}{e^{-(n+1)h}} \right] }{\displaystyle  \sum_{|i|\leq m} \left(1-\frac{|i|}{m}\right)^2 (2n+1) } \\
& = \frac{6}{2m^2+1}\frac{\displaystyle  \sum_{|j|\leq n} e^{2jh}}{2n+1} + \frac{2}{e^h+1}  \left[ \frac{2n}{2n+1} \left(e^{h/2}-1\right)^2 + \frac{1+e^h}{2n+1} \right]
\end{aligned}
\]
Taking $n$ large enough and then $m\to \infty$, we obtain that the above quantity approaches:
\[\frac{2}{e^h+1}\left(e^{h/2}-1\right)^2 = C_h (\rho(h))^2\]
and the conclusion follows.
\end{proof}

\section{Convergence of semi-discrete and discrete finite-difference schemes}\label{sec:semi-discret}

The aim of this section is twofold. First, we collect the properties of the two discrete Laplace operators on $\HH$ that will lead to the convergence of the finite-difference numerical scheme to the solution of the continuous problem \eqref{eq:heat-hyperbolic-source}. In the second part, we delve into the proof of the convergence result for the semi-discrete scheme corresponding to an abstract discrete Laplacian satisfying the aforementioned properties. We chose to consider this general setting since it provides sufficient conditions for the convergence of the semi-discrete numerical scheme associated to other discrete grids and Laplace operators one might further develop.

\subsection{An abstract semi-discrete numerical scheme for the heat equation on $\HH$ with source}

Let $(\CC_h^{i,j})_{(i,j)\in Z_1\times Z_2}$ the cells of a numerical grid with parameter $h>0$ on the half-plane model of $\HH$, where $Z_1,Z_2\subseteq \Z$ might depend on $h$. We also introduce the function space $\ell_h^2$ and
projection operator $\Pi_h$ corresponding to this abstract grid, as in Section \ref{sec:construction-first-lapl}. For functions in $\ell_h^2$ we also consider an abstract discrete finite-difference Laplace operator $\Delta_h$ and assume that the following properties are satisfied:
\begin{enumerate}[(L1)]
    \item (Contractivity of the projection) \label{L:contractivity} For every $v\in {\rm L}^2(\HH)$, $\|\Pi_h v\|_{\ell_h^2} \leq \|v\|_{{\rm L}^2(\HH)}$.
    
    \item (Dissipativity) \label{L:dissipativity} The operator $\Delta_h$ is self-adjoint on the Hilbert space $\ell_h^2$, with dense domain $D(\Delta_h)=\{v\in \ell_h^2: \Delta_h v \in \ell_h^2 \}$. Moreover, $\Delta_h$ is a m-dissipative operator in the sense of \cite[Chapter 2]{cazenave}.
    
    \item \label{L:consistency} (Consistency) There exists $h_0>0$ such that for every $u_0$ and $f$ satisfying Hypothesis~\ref{hyp:biserica-Grivitei}, and every time span $T>0$, there exists a constant $C_{u_0,f,T}>0$,  such that the solution $u$ the equation~\eqref{eq:heat-hyperbolic-source} satisfies the following:
    \begin{equation} \label{eq:consistency-discrete-lapl}
    \|\Delta_h \Pi_h e^{t\Delta_g} u(t) - \Pi_h \Delta_g e^{t\Delta_g} u(t)\|_{\ell_h^2}\leq C_{u_0,f,T}\, h^2,\, \forall h\in (0,h_0), t\in [0,T].
    \end{equation}
\end{enumerate}
\begin{rem}\label{rem:well-posedness-L2-estimate-discrete-heat}
For $u_0$ and $f$ satisfying Hypothesis~\ref{hyp:biserica-Grivitei}, let us consider the semi-discrete approximation of~\eqref{eq:heat-hyperbolic-source}:
\begin{eqnarray}\label{eq:heat-discrete-source}
\left\{
\begin{array}{ll}
\partial_t u^h_{i,j}(t) = (\Delta_h u^h(t))_{i,j} + (\Pi_h (f(t)))_{i,j}\,,
& t \in (0, T],\; (i,j)\in Z_1\times Z_2;\\[4pt]
\phantom{\partial_t} u^h_{i,j}(0) = (\Pi_h(u_0))_{i,j},
& \phantom{ t \in (0, \infty),}\;\; (i,j)\in Z_1\times Z_2\,,
\end{array}
\right.
\end{eqnarray}
where the initial data is projected on $\ell_h^2$. Properties (L\ref{L:contractivity}) and (L\ref{L:dissipativity}), alongside Duhamel's formula {\cite[Proposition 4.1.6]{cazenave}}, imply that the Cauchy problem stated above is \emph{well-posed} and the solution takes the form:
\begin{equation}\label{eq:soultion-heat-semidiscrete-source}
u^h(t)=e^{t\Delta_h} u_0+\int_0^t e^{(t-\sigma)\Delta_h} \Pi_h f(\sigma) \dd{\sigma}.
\end{equation}
\end{rem}

\begin{rem}
The results in Sections \ref{sec:first-lapl} and \ref{sec:second-lapl} imply that the two variants of restricted discrete Laplace operators $\Delta_{h,D}^{(1)}$ and $\Delta_{h,D}^{(2)}$ defined on the corresponding finite-dimensional $\ell_{h,D}^2$ spaces satisfy the properties (L\ref{L:contractivity})-(L\ref{L:consistency}), for the particular choice of $D\coloneqq D_{h,\gamma,\zeta}$ as in \eqref{Intro:eq:cine-PLM-i-D}.
\end{rem}

\subsection{Convergence of the semi-discrete scheme and error estimates}

The following theorem characterise the approximation error between the semi-discrete scheme \eqref{eq:heat-discrete-source} and the heat equation with source \eqref{eq:heat-hyperbolic-source}.

\begin{thm}[Convergence of the FDM scheme] Let $(\CC_h^{i,j})_{(i,j)\in Z_1\times Z_2}$ a grid on $\HH$ and $\Delta_h$ a discrete finite-difference operator associated to it, satisfying properties (L\ref{L:contractivity})-(L\ref{L:consistency}) above. Then, for every initial datum $u_0$ and $f$ satishying Hypothesis~\ref{hyp:biserica-Grivitei}, there exists a constant $C_{u_0,f,T}>0$ such that, for every $h\in (0,h_0)$ and $t\in [0,T],$
\[\|u^h(t)-\Pi_h u(t)\|_{\ell_h^2}\leq C_{u_0,f,T}\, h^2, \]
where $u^h$ is the solution of \eqref{eq:heat-discrete-source} and $u$ is the solution of \eqref{eq:heat-hyperbolic-source}.
\end{thm}
\begin{proof}
We take the time derivative of the $\ell_h^2$ norm of the error
\[E_h(t)\coloneqq u^h(t)-\Pi_h u(t)\]
and, using the form of the equations satisfied by $u$ and $u^h$, we obtain:
\[\begin{aligned}
\frac{\dd{}}{\dd{t}} \|E_h(t)\|^2_{\ell_h^2} & = 2 \left\langle u^h(t)-\Pi_h u(t), \Delta_h u^h(t) + \Pi_h f(t) - \Pi_h (\Delta_g u(t) + f(t))  \right\rangle_{\ell_h^2} \\
& = 2 \left\langle u^h(t)-\Pi_h u(t), \Delta_h u^h(t) - \Pi_h \Delta_g u(t)  \right\rangle_{\ell_h^2}\\
& =  2 \left\langle u^h(t)-\Pi_h u(t), \Delta_h u^h(t) - \Delta_h \Pi_h u(t)  \right\rangle_{\ell_h^2} + 2 \left\langle u^h(t)-\Pi_h u(t), \Delta_h \Pi_h u(t) - \Pi_h \Delta_g u(t)  \right\rangle_{\ell_h^2}\\
&= 2 \left\langle E_h(t), \Delta_h E_h(t) \right\rangle_{\ell_h^2} + 2 \left\langle E_h(t),  \Delta_h \Pi_h u(t) - \Pi_h \Delta_g u(t)\right\rangle_{\ell_h^2}.
\end{aligned}
\]
The dissipativity of the operator $\Delta_h$ (hypothesis (L\ref{L:dissipativity})) implies that the first term above is non-positive. Moreover, by the Cauchy-Schwarz inequality, we obtain:
\[\frac{\dd{}}{\dd{t}} \|E_h(t)\|^2_{\ell_h^2} \leq 2 \|E_h(t)\|_{\ell_h^2} \|\Delta_h \Pi_h u(t) - \Pi_h \Delta_g u(t)\|_{\ell_h^2},\]
which immediately leads to: 
\[\|E_h(t)\|_{\ell_h^2} \leq \|E_h(0)\|_{\ell_h^2} + \int_0^t \|\Delta_h \Pi_h u(s) - \Pi_h \Delta_g u(s)\|_{\ell_h^2} \dd{s}.\]
By the choice of initial data $u^h(0)$ of the numerical scheme, $E_h(0)=0$. This fact, combined with the consistency estimate (L\ref{L:consistency}) implies:
\[\|E_h(t)\|_{\ell_h^2} \leq \int_0^t \|\Delta_h \Pi_h u(s) - \Pi_h \Delta_g u(s)\|_{\ell_h^2} \dd{s}\leq C_{u_0,f,T} h^2,\]
which finishes the proof.
\end{proof}

\section{Fully discrete $\theta$-scheme for discrete Laplacians on $\HH$}\label{sec:crac-stix}
This section is dedicated to the study of the actual algorithm (Algorithm \ref{alg:theta-heat-hyperbolic}) that can be implemented on a computer in order to approximate the solution of the heat equation on the whole hyperbolic space $\HH$. Throughout this section, the tuple $((\mathcal{C}_h^{i,j})_{(i,j)\in Z_1\times Z_2},\ell_h^2,\Pi_h,\Delta_h)$ is either of the two finite grids, i.e. $C_h^{i,j}\subset \HHD$, functions spaces, i.e. $\ell_h^2 \coloneqq \ell^2_{h, D}$, projections, i.e. $\Pi_h \coloneqq \Pi_{h, D}$, and discrete Laplace operators, i.e. $\Delta_h \coloneqq \Delta_{h, D}$, in Sections \ref{sec:first-lapl} or \ref{sec:second-lapl}, respectively, for $D\coloneqq D_{h,\gamma,\zeta}$ as in \eqref{Intro:eq:cine-PLM-i-D}. 

The novel aspect is that the time is also discretised at a uniform time step equal to $\tau>0$ that will be chosen according to the space grid parameter $h$. Eventually, we will prove that the convergence rate of the resulting $\theta$-scheme for $\theta\in (\frac 1 2,1]$ is $\mathcal{O}(h^2+\tau)$, whereas, in the particular case $\theta=\frac{1}{2}$ corresponding to the Crank-Nicolson scheme, the convergence rate is $\mathcal{O}(h^2+\tau^2)$. In consequence, we select the time step $\tau$ according to $\theta$, on the lines \ref{alg:ifTheta}-\ref{alg:elseTheta} of Algorithm \ref{alg:theta-heat-hyperbolic}, in order to ensure a convergence rate of $\mathcal{O}(h^2)$ for our $\theta$-scheme, regardless of $\theta \in [\frac 1 2, 1]$.\\

From the numerically point of view, it is easier to transfer information from the continuous setting to the grid via pointwise evaluations, instead of using the integral projection operators $\Pi_h$ in \eqref{eq:projection-first-lapl} and \eqref{eq:projection-2}. The following lemma identifies the points $\xi_{i,j}$ in every cell $\CC_h^{i,j}$ which allow us to transfer the initial datum and source towards the grid (as on line \ref{alg:initial-datum} of Algorithm \ref{alg:theta-heat-hyperbolic}), while preserving the convergence rate $\mathcal{O}(h^2)$.
\begin{lem}\label{lem:approx-projection-with-pointwise}
Let $v\in \mathcal{M}$ and the tuple $((\mathcal{C}_h^{i,j})_{(i,j)\in Z_1\times Z_2},\ell_h^2,\Pi_h)$ be any of the two grids, functions spaces and projections in Sections \ref{sec:first-lapl} or \ref{sec:second-lapl}. Then, there exists a universal constant $C>0$ such that, uniformly in $h\in (0,1)$, the following estimate holds:
\begin{equation}\label{eq:approx-projection-with-pointwise}
    \|\Pi_h v - v(\xi_{i,j})\|_{\ell_h^2}\leq C h^2 \|v\|_{\mathcal M},
\end{equation}
where the point $\xi_{i,j}\in C_h^{i,j}$ satisfies:
\begin{equation}\label{eq:xi-is-mass-center}
\int_{C_h^{i,j}} (\x-\xi_{i,j}) \dd{\mu(\x)} = 0.
\end{equation}
More precisely, for the first discrete Laplace operator (Section \ref{sec:first-lapl}), 
\[\xi_{i,j}^{(1)}=\left(ih, h\left(j^2-\frac{1}{4}\right)\log\left(\frac{j+\frac{1}{2}}{j-\frac{1}{2}}\right)\right)\]
and for the second discrete Laplacian (Section \ref{sec:second-lapl}), 
\[\xi_{i,j}^{(2)} = \left(i \rho(h), \frac{h}{\rho(h)} e^{jh}\right).\]
\end{lem}

\begin{algorithm}\label{algorithm}
\DontPrintSemicolon
\KwIn{$u_0$, $f$, $T$, $h$, $\theta$, $\gamma$, $\zeta$, $\xi_{i,j}$ as in Lemma \ref{lem:approx-projection-with-pointwise}}
\KwOut{$U\sim u(t)$}
Initialisation: $D\leftarrow \zeta h^{-\gamma}$\;
\hspace{2.3cm} $Z_1\times Z_2 \leftarrow \left\{(i,j)\in \Z^2 : \xi_{i,j}\in [-D,D]\times [\frac 1 D,D]\right\}$\;
\hspace{2.3cm} $U_{i,j}\leftarrow u_0(\xi_{i,j})\,, \quad \forall~(i,j)\in Z_1\times Z_2$\;
\label{alg:initial-datum}
\hspace{2.3cm} $k\leftarrow 0$\;
\eIf{$\theta = \frac{1}{2}$ \label{alg:ifTheta}}{
$\tau\leftarrow h$\;
}{$\tau\leftarrow h^2$
\label{alg:elseTheta}
}

\Repeat{$k\tau > T$}{%
Update U:\;
$U\leftarrow ({\rm id}-\tau \theta \Delta_h)^{-1}\left[ \left({\rm id}+\tau (1-\theta)\right)U+ \tau f((k+\theta)\tau,\xi_{i,j}) \right] $ \;
\label{alg:update-u}
$k\leftarrow k+1$\;
}
\caption{$\theta$-scheme for the heat equation on $\HH$}
\label{alg:theta-heat-hyperbolic}
\end{algorithm}

\smallskip

The next theorem provides the convergence rate of Algorithm \ref{alg:theta-heat-hyperbolic}:
\begin{thm}[Convergence rate of the $\theta$-scheme]\label{thm:convergence-theta}
Let $T>0$, $u_0$ and $f$ satisfy Hypothesis~\ref{hyp:biserica-Grivitei} and assume furthermore that $\partial_t f, \partial_t^2 f, \partial_t \Delta_g f\in C([0,T],{\rm L}^2(\HH))$. Let also $\gamma>0$, $\zeta>2$, $\theta\in [\frac 1 2,1]$, $h\in (0,\frac 1 2)$, and $D$ as defined in~\eqref{Intro:eq:cine-PLM-i-D}. Then, there exists a constant $C_{T,u_0,f,\gamma,\zeta}>0$ such that, regardless of the choice of the tuple $((\CC_h^{i,j})_{(i,j)\in Z_1\times Z_2},\ell_h^2,\Pi_h,\Delta_h)$ as in the beginning of this section, the following estimate holds:
\[\left\|U-\Pi_h(u(T))_{i,j}\right\|_{\ell_h^2}\leq C_{T,u_0,f,\gamma,\zeta} (h^2+\tau),\]
where $U$ is the output vector of Algorithm \ref{alg:theta-heat-hyperbolic}.

Moreover, if $\theta=\frac{1}{2}$ (i.e., when we use the Crank-Nicolson scheme) and $T$ is an integral multiple of $\tau$, then the order of the convergence above is $\mathcal{O}(h^2+\tau^2)$.
\end{thm}

\section{Numerical resuls: A comparative study of the two discrete Laplacians}\label{sec:numerics}
This section is dedicated to the experimental study of the two variants of numerical scheme developed in this article, as an illustration for their practical use and also to demonstrate the sharpness of the convergence rate $\mathcal{O}(h^2)$. Additionally, we compare the numerical results obtained using these two schemes, emphasising the advantages of employing a grid specifically tailored to the geometry of the hyperbolic space. Finally, we deduce and test an analogous of the best performing  discrete Laplacian on the three dimensional hyperbolic space $\mathbb{H}^3$.

\subsection{Example setup}
We consider the heat equation~\eqref{eq:heat-hyperbolic-source} on the entire space $\HH$ with the following specified heat source:
\begin{equation}\label{eq:num-heat-source}
    f(t, \x) \coloneqq - \mathrm{e}^{- t - x_1^2 - x_2^2 - x_2^{-2}} \left[2 x_2^2 \left( 2 x_1^2 - 1\right) + 2 \dfrac{2 x_2^8 - x_2^6 - 4 x_2^4 - 3 x_2^2 + 2}{x_2^4} + 1\right]\,,
\quad t \in (0, +\infty),~\x \in \HH\,,
\end{equation}
and the following initial temperature distribution:
\begin{equation}\label{eq:num-heat-initial-condition}
    u_0(\x) \coloneqq \mathrm{e}^{- x_1^2 - x_2^2 - x_2^{-2}}\,,
\quad \x \in \HH\,.
\end{equation}
It can be easily proved that the analytical solution of the heat equation~\eqref{eq:heat-hyperbolic-source} corresponding to the previously defined heat source~\eqref{eq:num-heat-source} and initial condition~\eqref{eq:num-heat-initial-condition} is given by
\begin{equation}\label{eq:num-heat-solution}
    u^{(\mathrm{ex})}(t, \x)
= \mathrm{e}^{- t - x_1^2 - x_2^2 - x_2^{-2}}\,,
\quad t \in (0, +\infty),~\x \in \HH\,.
\end{equation}
This exact solution serves as the benchmark we aim to approximate using our proposed numerical schemes.

This study will proceed as follows:  for various values of the spatial grid parameter, specifically $h \in \{\frac{1}{16}, \frac{1}{32}, \frac{1}{64}\}$, we will numerically solve the investigated problem with the prescribed heat source~\eqref{eq:num-heat-source} and initial condition~\eqref{eq:num-heat-initial-condition}. Subsequently, we will assess and compare the $\ell^{2}_{h, D}$ norms of the errors in the numerical approximation of the exact solution~\eqref{eq:num-heat-solution} obtained using both numerical schemes.

Regarding the choice of the bounded approximation domain $\HHD$ in \eqref{Intro:eq:cine-PLM-i-HD}, for our specific example we choose $\zeta = 6$ and $\gamma = \frac{1}{6}$ resulting in the following values of $D \in \{9.52, 10.69, 12\}$ corresponding to each value of the spatial discretisation parameter $h \in \{\frac{1}{16}, \frac{1}{32}, \frac{1}{64}\}$. See Figure \ref{fig:NRE-Comparison}.

\subsection{Convergence Order of the FDM}

Let $U^{(\ell)}_h$ be the output vector obtained using the $\theta$-scheme given by Algorithm~\ref{alg:theta-heat-hyperbolic}, where $\ell = 1, 2$ denotes the discrete Laplacian employed in the numerical approximation and $h > 0$ the step size of the spatial grid. To study the convergence of the aforementioned schemes towards the exact solution, we denote, for a fixed time $T=1$, the following convergence errors with respect to the $\ell^{2}_{h, D}$ norm
\[
    E^{(\ell)}_h \coloneqq \left\|U^{(\ell)}_h-u(T, \xi_{i,j})\right\|_{\ell^{2}_{h, D}}, \quad \ell = 1, 2.
\]

In Tables~\ref{table01} and \ref{table02} we present the values of the convergence errors $E^{(\ell)}_h$ ($\ell = 1, 2$), the number of FDM nodes needed to discretise $\HHD$, i.e. $N \times M$, the memory usage and the CPU time\footnote{The computation was performed on an Intel(R) Core(TM) i7-6700HQ CPU @ 2.60GHz processor.} required by Algorithm~\ref{alg:theta-heat-hyperbolic} with $\theta = \frac 1 2 $ and $\theta = 1$, respectively, for various spatial step sizes, $h \in \{\frac{1}{16}, \frac{1}{32}, \frac{1}{64}\}$. The following conclusions can be drawn from these tables:

\begin{enumerate}[(i)]
\setlength{\itemsep}{1pt}
\item For each discrete Laplacian and each choice of $\theta$, the convergence errors $E^{(\ell)}_h$ ($\ell = 1, 2$) decrease by a factor of $4$ as the mesh size $h$ is halved and the number of FDM nodes $N \times M$ increases. This indicates that the $\theta$-scheme in Algorithm~\ref{alg:theta-heat-hyperbolic} is convergent with a sharp order of $\mathcal{O}(h^2)$ with respect to the refinement of the finite difference grid.

\item For each discrete Laplacian and each $h$, we observe that the $\theta$-scheme corresponding to $\theta = \frac{1}{2}$ is much faster than the one corresponding to $\theta = 1$, as expected, since in the former case we were able to chose the time step $\tau = h$, in contrast to $\tau = h^2$ as in the latter case.

\item For each $h$ and each choice of $\theta$, the convergence errors $E^{(\ell)}_h$ ($\ell = 1, 2$) associated with the second discrete Laplacian are significantly smaller than their counterparts obtained using the first discrete Laplacian, illustrating the advantage of using a grid that is specifically tailored to the geometry of $\HH$. 

\item For each $h$ and each choice of $\theta$, the second discrete Laplacian requires significantly fewer FDM nodes compared to the first discrete Laplacian, resulting in substantial improvements in memory usage and CPU time.
\end{enumerate}

\begin{table}[ht]
\normalsize
\begin{center}
\begin{tabular}{|c|c|c|c|c|c|}
\hline $\Delta_h$ & $h$ & $E^{(\ell)}_h$ & $N \times M$ & Memory usage (MB) & CPU Time (s)\\
\hline \multirow{3}{2em}{$\Delta_h^{(1)}$} 
& $1 / 16$ & $2.3153\mathrm{e}{-04}$ & $\phantom{10}46360$ & \phantom{1}108 & $\phantom{1}0.29$\\  \cline{2-6}
& $1 / 32$ & $5.4069\mathrm{e}{-05}$ & $\phantom{1}233585$ & \phantom{1}430 & $\phantom{1}3.40$\\  \cline{2-6}
& $1 / 64$ & $1.3314\mathrm{e}{-05}$ & $1174268$ & 2206 & $43.75$\\
\hline \multirow{3}{2em}{$\Delta_h^{(2)}$} 
& $1 / 16$ & $1.2119\mathrm{e}{-04}$ & $\phantom{10}22265$ & \phantom{10}24 & $\phantom{1}0.14$\\  \cline{2-6}
& $1 / 32$ & $3.0210\mathrm{e}{-05}$ & $\phantom{1}103435$ & \phantom{1}144 & $\phantom{1}1.31$\\  \cline{2-6}
& $1 / 64$ & $7.5480\mathrm{e}{-06}$ & $\phantom{1}489665$ & \phantom{1}792 & $16.86$\\
\hline
\end{tabular}
\caption{Values of the convergence errors, $E^{(\ell)}_h$ ($\ell = 1, 2$), the corresponding number of FDM nodes used in the discretisation of $\HHD$, i.e. $N \times M$, the memory usage and the CPU time, corresponding to Algorithm~\ref{alg:theta-heat-hyperbolic} with $\theta = \frac 1 2$ and $h \in \{\frac{1}{16}, \frac{1}{32}, \frac{1}{64}\}$.}\label{table01}
\end{center}
\end{table}

\begin{table}[ht]
\normalsize
\begin{center}
\begin{tabular}{|c|c|c|c|c|c|}
\hline $\Delta_h$ & $h$ & $E^{(\ell)}_h$ & $N \times M$ & Memory usage (MB) & CPU Time (s)\\
\hline \multirow{3}{2em}{$\Delta_h^{(1)}$} 
& $1 / 16$ & $2.1438\mathrm{e}{-04}$ & $\phantom{10}46360$ & \phantom{10}91 & $\phantom{100}4.06$\\  \cline{2-6}
& $1 / 32$ & $5.1647\mathrm{e}{-05}$ & $\phantom{1}233585$ & \phantom{1}366 & $\phantom{1}106.48$\\  \cline{2-6}
& $1 / 64$ & $1.3193\mathrm{e}{-05}$ & $1174268$ & 1860 & $2524.66$\\
\hline \multirow{3}{2em}{$\Delta_h^{(2)}$} 
& $1 / 16$ & $1.6205\mathrm{e}{-04}$ & $\phantom{10}22265$ & \phantom{10}64 & $\phantom{100}1.89$\\  \cline{2-6}
& $1 / 32$ & $4.0443\mathrm{e}{-05}$ & $\phantom{1}103435$ & \phantom{1}177 & $\phantom{10}40.87$\\  \cline{2-6}
& $1 / 64$ & $1.0108\mathrm{e}{-05}$ & $\phantom{1}489665$ & \phantom{1}795 & $1016.55$\\
\hline
\end{tabular}
\caption{Values of the convergence errors, $E^{(\ell)}_h$ ($\ell = 1, 2$), the corresponding number of FDM nodes used in the discretisation of $\HHD$, i.e. $N \times M$, the memory usage and the CPU time, corresponding to Algorithm~\ref{alg:theta-heat-hyperbolic} with $\theta = 1$ and $h \in \{\frac{1}{16}, \frac{1}{32}, \frac{1}{64}\}$.}\label{table02}
\end{center}
\end{table}

\begin{figure}[ht]
\centering
\SetFigLayout{3}{2}
\subfigure[$\theta = \frac 1  2$]{
\includegraphics[scale=0.2]
{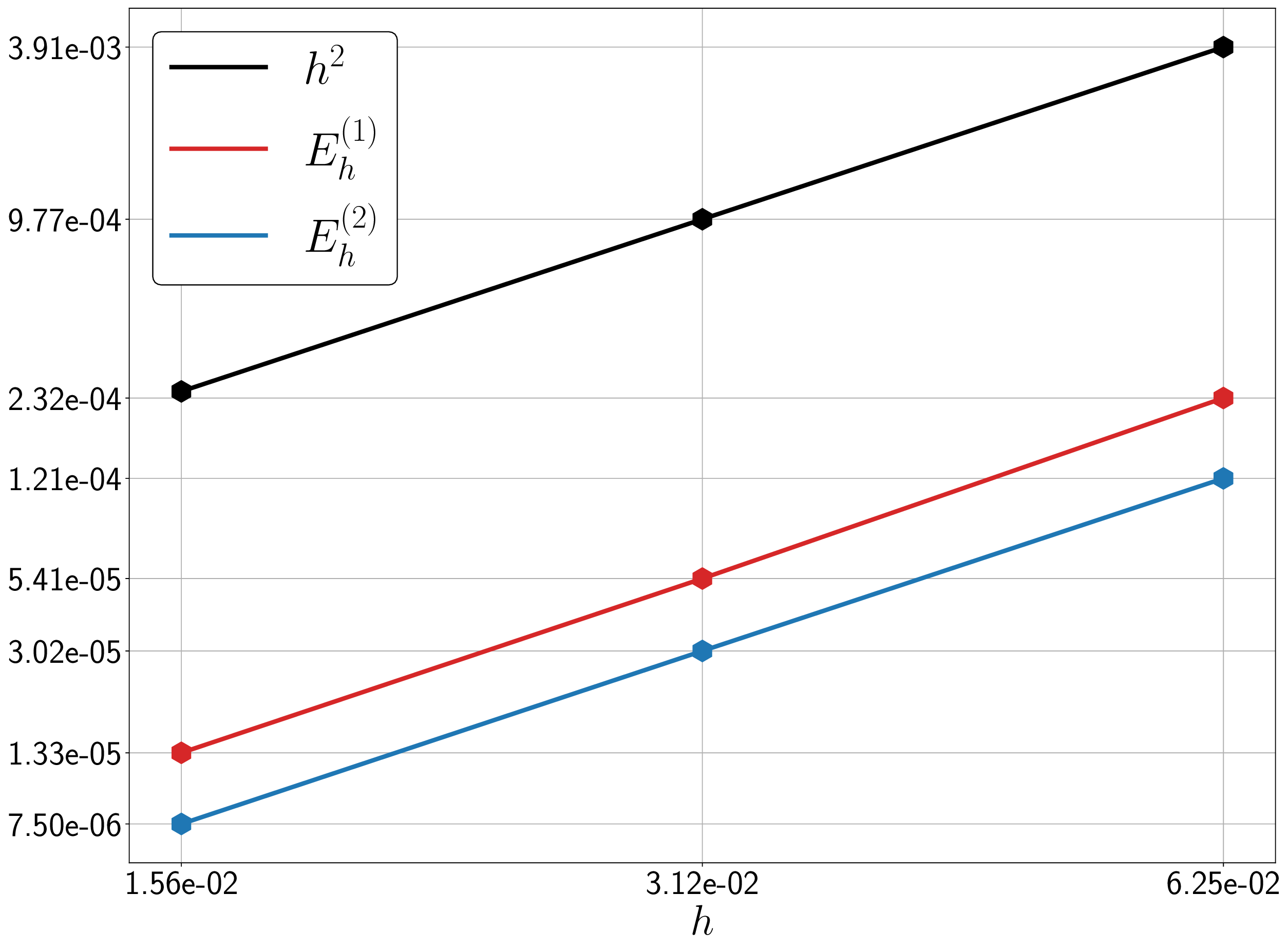}}
\subfigure[$\theta = 1$]{
\includegraphics[scale=0.2]
{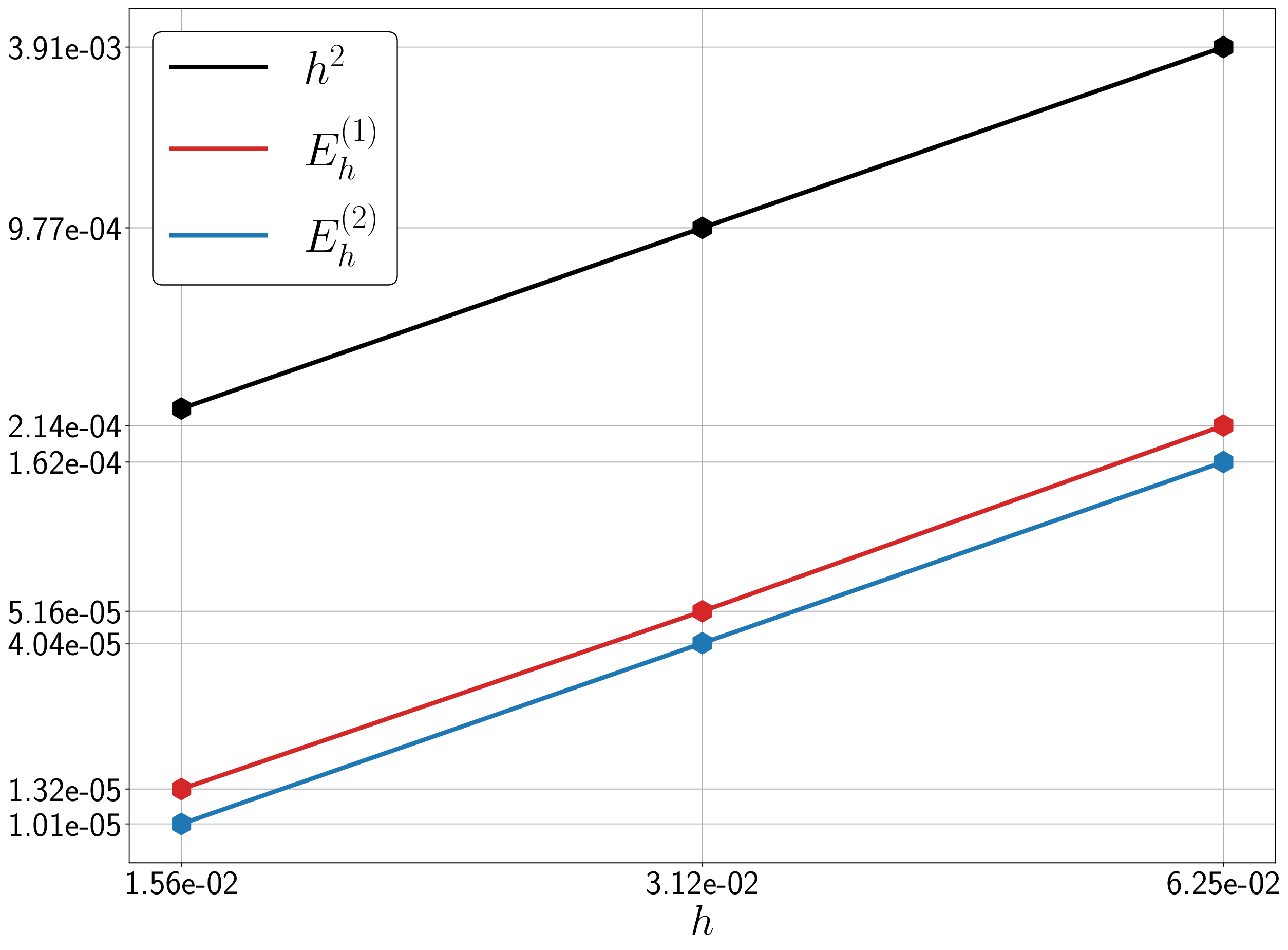}}
\hspace*{0.1cm}
\caption{The $\ell^{2}_{h, D}$-norms of the errors in the numerical approximation yielded by the \textcolor{red}{first} and \textcolor{blue}{second} Laplacian, respectively, as functions of the spatial step size $h$, obtained using Algorithm~\ref{alg:theta-heat-hyperbolic} with (a)~$\theta = \frac 1  2$ and (b)~$\theta = 1$, respectively, for $h \in \{\frac{1}{16}, \frac{1}{32}, \frac{1}{64}\}$.}
\label{fig:Conv-Comparision}
\end{figure}

Further, Figures~\ref{fig:Conv-Comparision} (a) and~(b) display, on a logarithmic scale, the $\ell^{2}_{h, D}$ norms of the errors in the numerical approximation of the exact solution~\eqref{eq:num-heat-solution}, i.e., $E^{(\ell)}_h$, $\ell = 1, 2$, corresponding to the first and second discrete Laplacian, respectively, obtained using Algorithm~\ref{alg:theta-heat-hyperbolic} with $\theta = \frac{1}{2}$ and $\theta = 1$, respectively, as functions of the spatial grid size $h$, as well as the graph of the function $h^2$. It can be seen from this figure that the three graphs mentioned above are parallel, meaning that both numerical FDM discretisations of the Laplacian on $\HH$ provide us with a numerical approximation for the exact solution with a sharp convergence rate of $\mathcal{O}(h^2)$. Moreover, the comparison of the $\ell^{2}_{h, D}$-norms of the errors associated with the first and second discrete Laplace operators reveals that the latter numerical scheme yields a more accurate approximation. This result underscores the effectiveness of utilising a grid specifically adapted to the geometry of $\HH$. Although not presented herein, similar results have been obtained when employing Algorithm~\ref{alg:theta-heat-hyperbolic} for various values for $\theta \in (\frac 1 2, 1]$, confirming the sharp convergence rate of $\mathcal{O}(h^2)$ for any choice of $\theta$.

\subsection{Normalised relative error of the approximation}
To investigate the precision of the pointwise approximation of the proposed algorithm, we define the normalised relative error in the numerically retrieved solution with respect to its analytical counterpart, namely
\begin{align*}
e^{(\ell)}_h
&\coloneqq \dfrac{\big\vert U^{(\ell)}_{h,i,j}
- u^{(\mathrm{ex})}(T, \xi_{i, j}) \big\vert}
{\displaystyle \max_{\y \in \HHD}
\big\vert u^{(\mathrm{ex})}(T,\y) \big\vert},
\quad (i,j)\in Z_1\times Z_2\,, \ell = \overline{1, 2},\, \xi_{i,j}\text{ as in Lemma \ref{lem:approx-projection-with-pointwise}.}
\end{align*}

\begin{figure}[htbp]
\centering
\SetFigLayout{3}{2}
\subfigure[$e^{(1)}_h$: First discrete Laplacian, $h = \frac{1}{ 16}$, $D = 9.52$]{
\includegraphics[scale=0.55]
{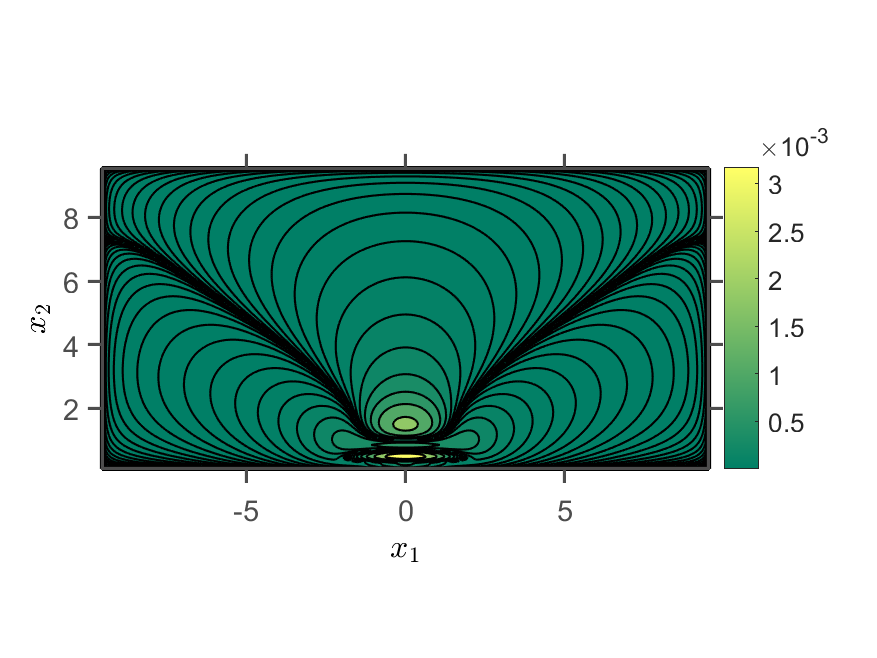}}
\vspace*{0.1cm}
\subfigure[$e^{(2)}_h$ Second discrete Laplacian, $h = \frac{1}{ 16}$, $D = 9.52$.]{
\includegraphics[scale=0.55]
{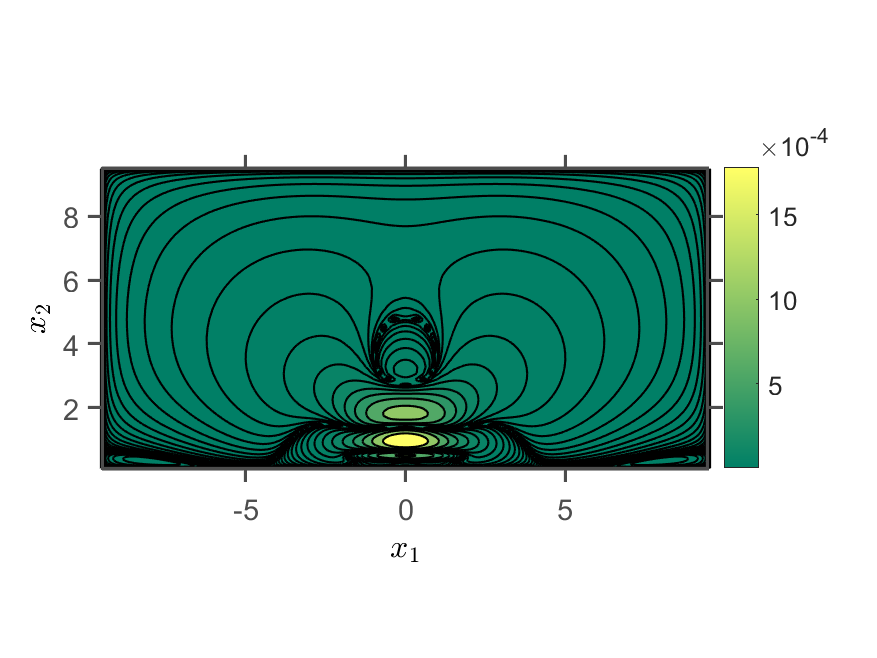}}
\hspace*{0.1cm}
\subfigure[$e^{(1)}_h$: First discrete Laplacian, $h = \frac 1 { 32}$, $D = 10.69$]{
\includegraphics[scale=0.55]
{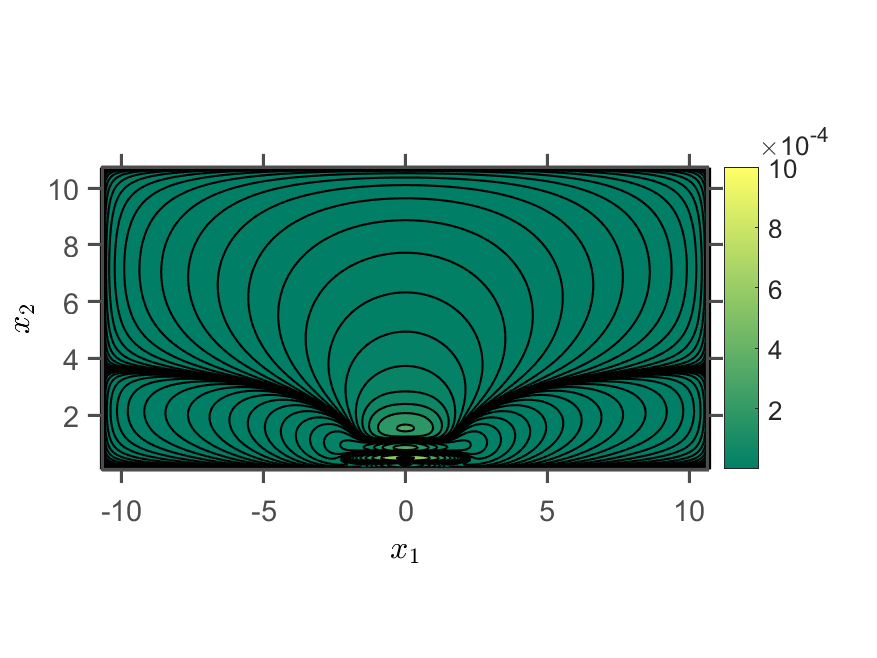}}
\vspace*{0.1cm}
\subfigure[$e^{(2)}_h$: Second discrete Laplacian, $h = \frac 1 { 32}$, $D = 10.69$]{
\includegraphics[scale=0.55]
{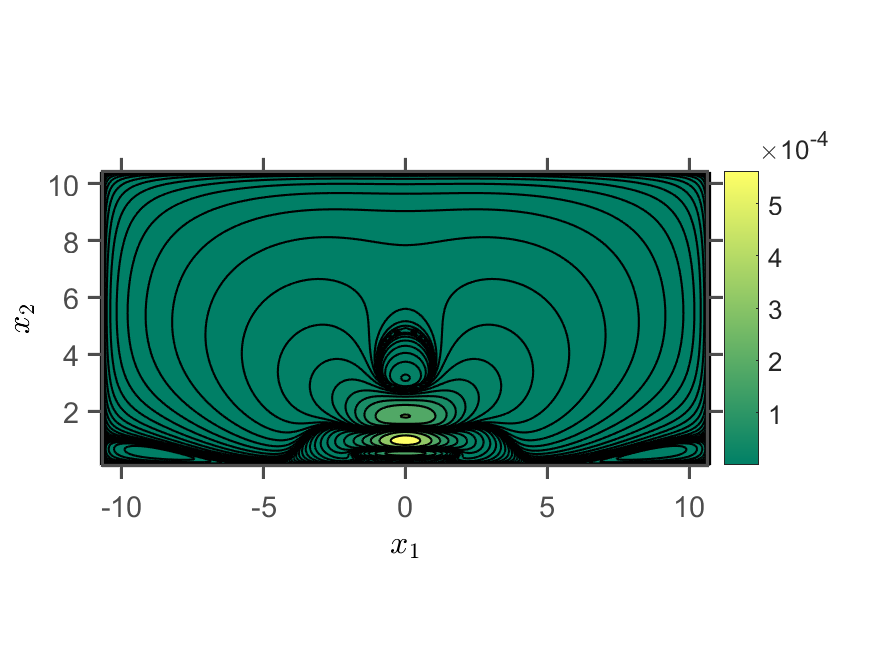}}
\hspace*{0.1cm}
\subfigure[$e^{(1)}_h$: First discrete Laplacian, $h = \frac{1} {64}$, $D = 12$]{
\includegraphics[scale=0.55]
{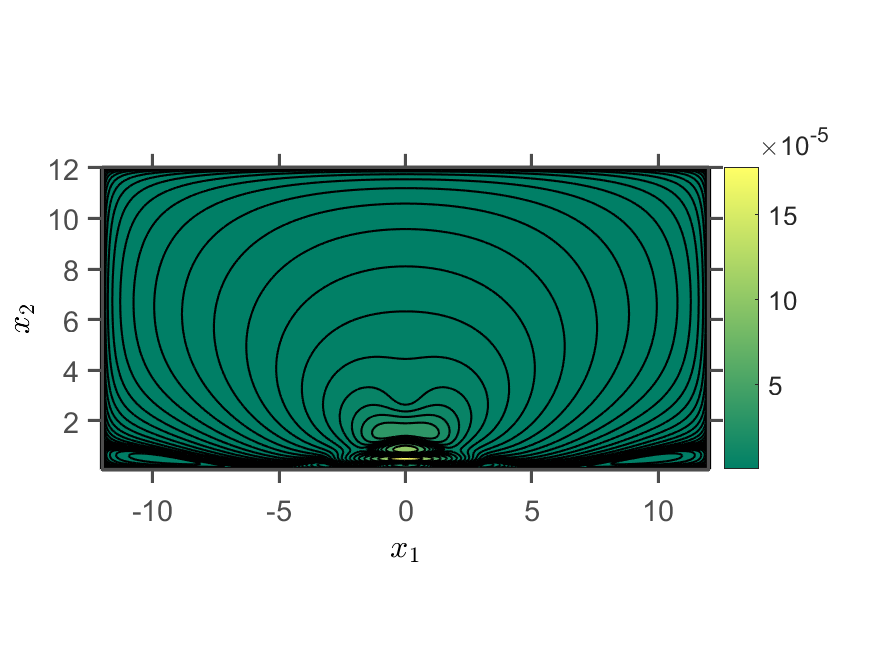}}
\vspace*{0.1cm}
\subfigure[$e^{(2)}_h$: Second discrete Laplacian, $h = \frac 1 { 64}$, $D = 12$]{
\includegraphics[scale=0.55]
{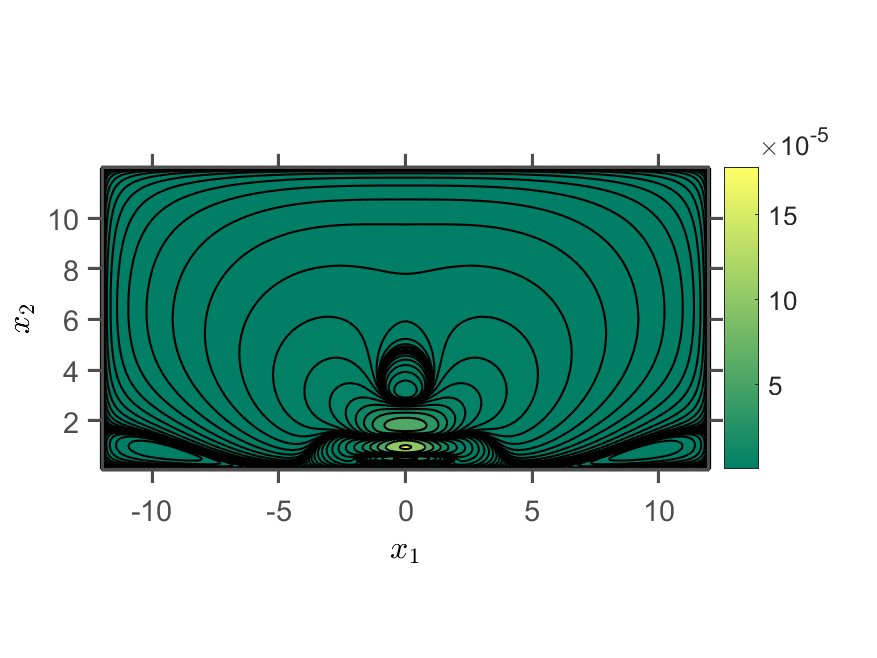}}
\hspace*{0.1cm}
\caption{The normalised relative errors $e^{(1)}_h$, corresponding to the numerical approximation obtained using the first discrete Laplacian are displayed on the left-hand side above, whereas the error terms $e^{(2)}_h$, corresponding to the numerical approximation obtained using the second discrete Laplacian are on the right-hand side. The results were obtained using Algorithm~\ref{alg:theta-heat-hyperbolic} with $\theta = \frac 1 2$ and $h \in \{\frac{1}{16}, \frac{1}{32}, \frac{1}{64}\}$.}
\label{fig:NRE-Comparison}
\end{figure}

Figure~\ref{fig:NRE-Comparison} displays a comparison between the normalised relative errors obtained by employing the first discrete Laplacian in the $\theta$-scheme given by Algorithm~\ref{alg:theta-heat-hyperbolic} corresponding to $\theta = \frac{1}{2}$, illustrated in Figures~\ref{fig:NRE-Comparison}~(a), (c) and (e), and their counterparts obtained via the second discrete Laplacian, depicted in Figures~\ref{fig:NRE-Comparison}~(b), (d) and (f), for various spacial step sizes, $h \in \{\frac{1}{16}, \frac{1}{32}, \frac{1}{64}\}$. From this figures it can be seen that Algorithm~\ref{alg:theta-heat-hyperbolic} provides us with a numerical approximation convergent with respect to refining the finite difference grid for both discrete Laplacians considered, noting how the normalised relative errors $e^{(\ell)}_h$, $\ell = 1, 2$, decrease as $h$ decreases.  Furthermore, by comparing Figures~\ref{fig:NRE-Comparison}~(a) and~(b) with Figures~\ref{fig:NRE-Comparison}~(c) and~(d), as well as with Figures~\ref{fig:NRE-Comparison}~(e) and~(f), it can be noted that the computational domain $\HHD$ increases as the stepsize $h$ decreases.

\subsection{Extension to 3D}
\label{sec:3D}

The finite difference method (FDM) scheme can be naturally extended to higher-dimensional hyperbolic spaces, $\mathbb{H}^n$ for $n \geq 3$. However, the memory and computational costs associated with the first discrete Laplacian increase drastically with the dimension, rendering it impractical for implementation beyond two dimensions. In contrast, the second discretisation remains computationally efficient and memory-friendly, making it well-suited for practical applications in three-dimensional hyperbolic space. The expression of the (continuous) Laplace-Beltrami operator in the half-space model of $\mathbb{H}^3$ is:
\[\Delta_g v(\x) = x_3^2 \Delta_e v(\x) - x_3 \partial_{x_3} v(\x),\text{ for } \x=(x_1,x_2,x_3)\in \mathbb{H}^3 \simeq \RR_+^3.\]
Its discrete counterpart follows by the same method used to deduce the 2D formula \eqref{eq:laplacian-2} and the resulting expression is:
\begin{equation}\label{eq:laplacian-2-3D}
    \begin{aligned} 
     (\Delta^{(2)}_h v^h)_{i,j, k} \coloneqq &\frac{1}{(\rho(h))^2}\Bigg[e^{2kh}(v^h_{i+1,j,k}+v^h_{i-1,j,k}+v^h_{i,j+1,k}+v^h_{i,j-1,k}-4 v^h_{i,j,k}) + \\
    & \left. \left(\frac{2}{e^h+1} - \dfrac{\rho(h)^2}{e^{2h} - 1} \right)v^h_{i,j,k+1} + \left(\frac{2 e^h}{e^h+1}  + \dfrac{\rho(h)^2 e^{2h}}{e^{2h} - 1} \right)v^h_{i,j,k-1} - \left(2 + \rho(h)^2 \right)v^h_{i,j, k}\right].
    \end{aligned}  
\end{equation}

To validate this scheme, we consider a benchmark solution analogous to the one used in the 2D case, but now focusing on a stationary heat problem on $\mathbb{H}^3$ by disregarding the time component. Similarly to the two-dimensional setting, assuming sufficient tail control, we approximate the equation on the whole space with a homogeneous Dirichlet boundary-value problem on a bounded domain 
\[\mathbb{H}^3_D \coloneqq [-D,D]\times[-D,D]\times \left[\frac 1 D,D\right], \,\text{ with }D\text{ as in \eqref{Intro:eq:cine-PLM-i-D}}.\]
Choosing a heat source similar to \eqref{eq:num-heat-source}, we prescribe the following exact solution:
\begin{equation}\label{eq:num-heat-solution-3D}
    v^{(\mathrm{ex})}(\x)
= \mathrm{e}^{- x_1^2 - x_2^2 - x_3^2 - x_3^{-2}}\,,
\quad ~\x \in \mathbb{H}^3\,.
\end{equation}

\noindent This setup enables a direct extension of the numerical experiment performed in the two-dimensional case and provides a basis for evaluating the convergence order of the proposed 3D scheme by evaluating the convergence error corresponding to the $\ell^{2}_{h, D}$-norm, 
\begin{equation}\label{eq:err-elliptic}
E_h^{(2)}\coloneqq \left\| v^{(ex)}(\xi_{i,j,k})-(\Delta_{h,D}^{(2)})^{-1}\left[\Delta_g v^{(ex)}(\xi_{i,j,k})\right]\right\|_{\ell_{h,D}^2}, \quad \xi_{i,j,k}\text{ center of }\CC_{h}^{i,j,k},
\end{equation}
for a varying spacial grid parameter, $h \in \{\frac{1}{8}, \frac{1}{16}, \frac{1}{32}\}$ and $D\coloneqq 2 h^{-\frac{1}{6}}$.

\begin{figure}[ht]
\centering
\subfigure[]{\includegraphics[scale=0.19]
{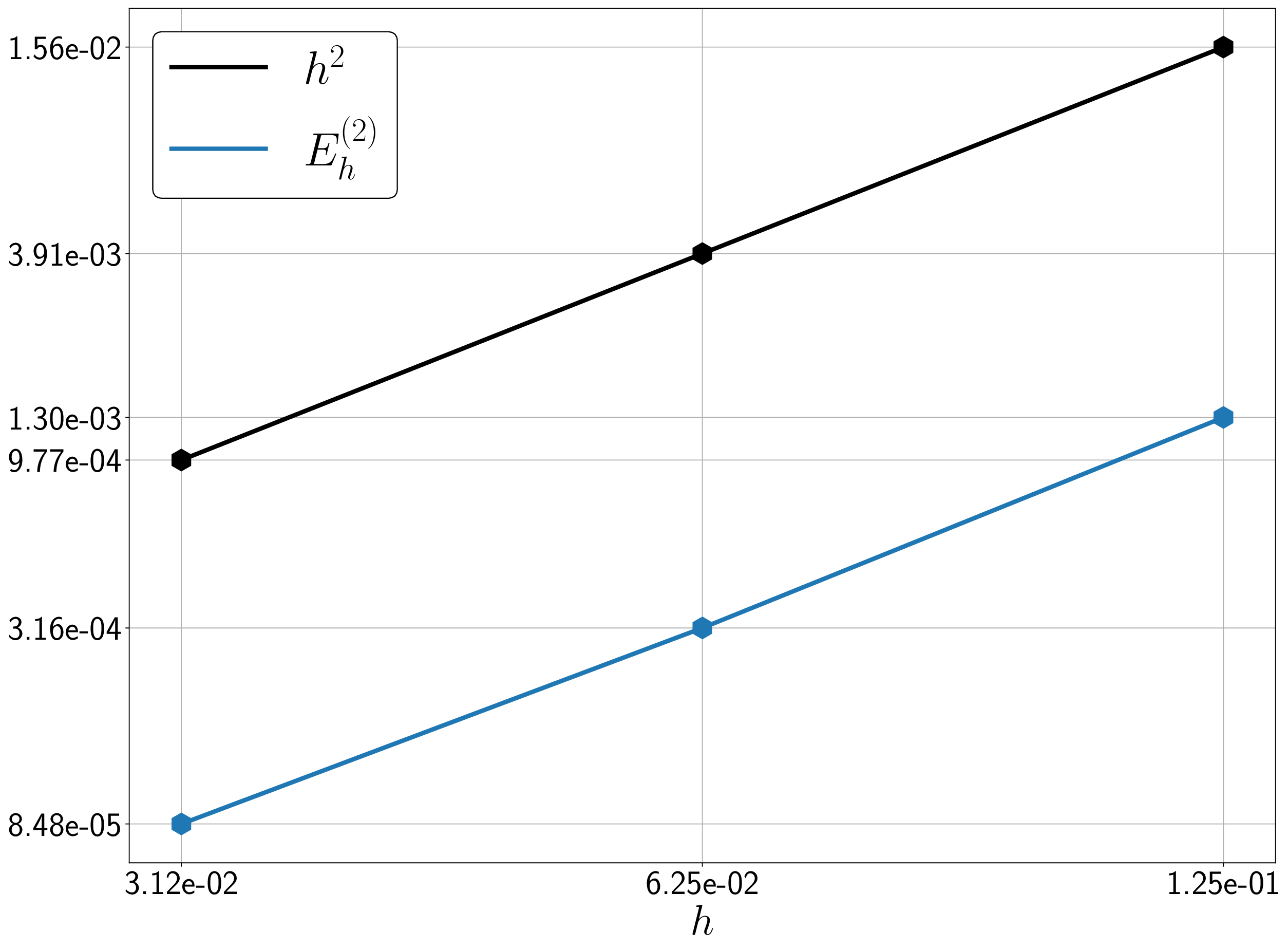}}
\hspace*{0.1cm}
\subfigure[]{\includegraphics[scale=0.3]
{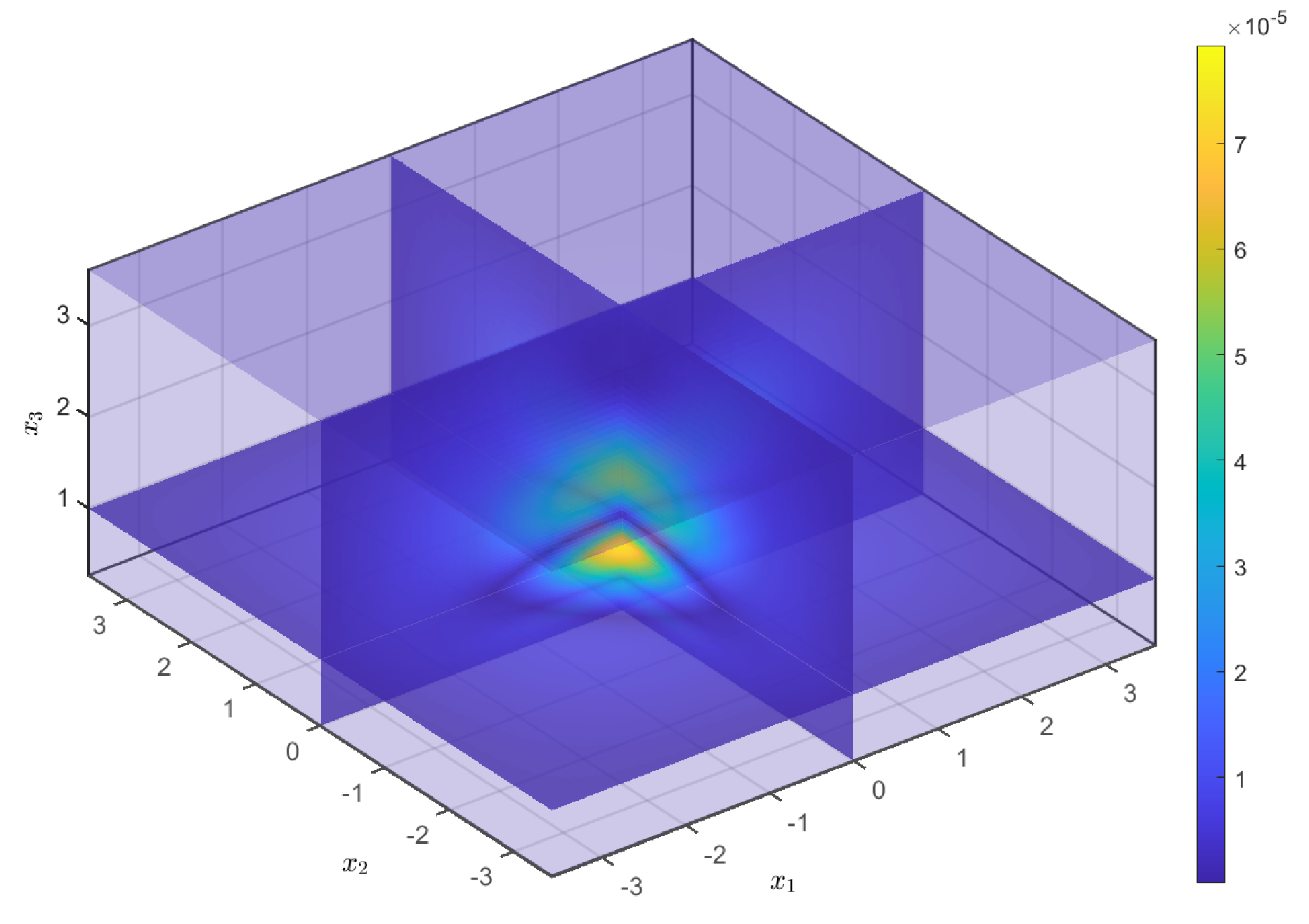}}
\caption{(a) The $\ell^{2}_{h, D}$-norm of the error \eqref{eq:err-elliptic} in the numerical approximation yielded by the 3D discrete Laplacian \eqref{eq:laplacian-2-3D}, as functions of the spatial step size $h$, for the stationary problem heat problem, with $h \in \{\frac{1}{8}, \frac{1}{16}, \frac{1}{32}\}$. (b) The normalised relative error for $h=\frac{1}{32}$, plotted as a heat map in the domain $\mathbb{H}^3_D$.}
\label{fig:Conv-Laplace-3D}
\end{figure}

Figure~\ref{fig:Conv-Laplace-3D}(a) shows, on a logarithmic scale, the $\ell^{2}_{h, D}$ error $E^{(2)}_h$ between the exact solution~\eqref{eq:num-heat-solution-3D} and the one obtained using the FDM scheme associated to the three-dimensional discrete Laplacian \eqref{eq:laplacian-2-3D} applied to the stationary problem. The error is plotted as a function of the spatial grid size $h$, alongside the reference graph of the function $h^2$. The two lines appear parallel, clearly indicating that the numerical scheme exhibits a second-order convergence rate, i.e., $\mathcal{O}(h^2)$. This confirms that the proposed finite difference discretisation technique for the Laplacian on $\mathbb{H}^n$ yields sharp and consistent convergence behavior, even in the three-dimensional setting.

\section{Conclusion and further research directions}
\label{sec:conclusion}

\subsection{Concluding remarks}

In this paper, we have constructed two discrete finite-difference approximations of the Laplace-Beltrami operator on the 2-dimensional hyperbolic space $\HH$, providing us with consistent, stable and thus convergent numerical schemes approximating the heat equation~\eqref{eq:heat-hyperbolic-source}. 

The consistency order of the aforementioned approximations has been proven to be $\mathcal{O}(h^2)$, see Remarks~\ref{rem:consistency-first-laplace-sos} and~\ref{rem:consistency-second-laplace-sos}, for the first and second discrete Laplacians, respectively. By employing a $\theta$-scheme to discretise the time variable, we propose Algorithm~\ref{alg:theta-heat-hyperbolic} to provide a numerical approximation of the solution of~\eqref{eq:heat-hyperbolic-source} with a convergence order of $\mathcal{O}(h^2)$ for both discrete Laplacians, see Theorem~\ref{thm:convergence-theta}. 

The sharpness of this theoretical results is clearly illustrated by the numerical experiments conducted in Section~\ref{sec:numerics}, see Figure~\ref{fig:Conv-Comparision}. Moreover, the suitability of the second discrete Laplace to the geometry of the hyperbolic space is emphasised from both a theoretical and a numerical point of view, see Remark~\ref{rem:second-laplace-greutati} and, Tables~\ref{table01}--\ref{table02}, respectively. This better-performing discrete Laplacian can be effectively generalized to higher dimensions, refer to Section \ref{sec:3D}

It is worth noting that the numerical solution provided by Algorithm~\ref{alg:theta-heat-hyperbolic} is an approximation to the heat equation on the entire space $\HH$. This has been achieved by defining the bounded domain $\HHD$ in terms of the mesh size, $h$, see~\eqref{Intro:eq:cine-PLM-i-HD}, resulting in the enlargement of $\HHD$ with respect to the decrease of $h$, as it can be seen in Figure~\ref{fig:NRE-Comparison}.

Furthermore, both variants of discrete Laplace operator on $\HH$ are strictly negative and their coercivity constants approach $\frac{1}{4}$ (i.e., the Poincar\' e optimal constant in the continuous setting), as $h$ approaches zero, refer to Theorems \ref{thm:first-laplacian-poincare} and  \ref{thm:lapl-2-poincare}.
\subsection{Future work directions}

We consider that this paper opens new research directions related to the problem of constructing a discrete counterpart to the Laplace-Beltrami operator on Riemannian manifolds. Some further research ideas related to the subject of this article are listed below in a synthetic manner.

\begin{enumerate}[I)]
    \item Produce other grids and corresponding discrete Laplacians better suited to the hyperbolic geometry, possibly using other models of the hyperbolic space. This is indeed a challenging problem, see Remark \ref{rem:equidistant-complicated}. 
    \item Study whether the discrete Laplace operator constructed in \cite{BuragoIvanov2014} using embedded graphs does fit into the framework of Sections \ref{sec:semi-discret} and \ref{sec:crac-stix}, in terms of yielding a convergent numerical scheme.
    \item Use the ideas in \cite{YuDeSa2019} to represent more accurately the numerical values of the grid functions in the computer's memory and analyse whether there is a significant improvement in the accuracy of the approximation of \eqref{eq:heat-hyperbolic-source} using the FDM schemes developed in the present paper.
    \item Inspired by the literature pertaining to artificial boundary conditions for the heat equation in unbounded domains of the Euclidean space \cite{WuSun2004,PangYang2017,ZhengXie2023}, produce more accurate artificial conditions at the boundary of the bounded domian $\HHD$ using the explicit solution of the heat equation on $\HH$.
\end{enumerate}

\printbibliography

@book{cazenave,
  title={An Introduction to Semilinear Evolution Equations},
  author={Cazenave, T. and Haraux, A.},
  year={1998},
  publisher={Oxford University Press}
}

@article {sharpPoincare2018,
    AUTHOR = {Nguyen, Van Hoang},
     TITLE = {The sharp {P}oincar\'{e}-{S}obolev type inequalities in the
              hyperbolic spaces {$\Bbb{H}^n$}},
   JOURNAL = {J. Math. Anal. Appl.},
  FJOURNAL = {Journal of Mathematical Analysis and Applications},
    VOLUME = {462},
      YEAR = {2018},
    NUMBER = {2},
     PAGES = {1570--1584},
      ISSN = {0022-247X,1096-0813},
   MRCLASS = {46E40},
  MRNUMBER = {3774305},
MRREVIEWER = {Pradipta\ Bandyopadhyay},
       DOI = {10.1016/j.jmaa.2018.02.054},
       URL = {https://doi.org/10.1016/j.jmaa.2018.02.054},
}

@article{DaviesMandouvalos,
author = {Davies, E. B. and Mandouvalos, N.},
title = {Heat Kernel Bounds on Hyperbolic Space and Kleinian Groups},
journal = {Proceedings of the London Mathematical Society},
volume = {s3-57},
number = {1},
pages = {182-208},
doi = {https://doi.org/10.1112/plms/s3-57.1.182},
url = {https://londmathsoc.onlinelibrary.wiley.com/doi/abs/10.1112/plms/s3-57.1.182},
eprint = {https://londmathsoc.onlinelibrary.wiley.com/doi/pdf/10.1112/plms/s3-57.1.182},
year = {1988}
}

@article{GrigoryanHeatKernelOnHyperbolic,
author = {Grigor'yan, Alexander and Noguchi, Masakazu},
title = {The Heat Kernel on Hyperbolic Space},
journal = {Bulletin of the London Mathematical Society},
volume = {30},
number = {6},
pages = {643-650},
doi = {https://doi.org/10.1112/S0024609398004780},
url = {https://londmathsoc.onlinelibrary.wiley.com/doi/abs/10.1112/S0024609398004780},
year = {1998}
}

@misc{david2022DiscreteHardyLine,
      title={An improved discrete Rellich inequality on the half-line}, 
      author={Borbala Gerhat and David Krejcirik and Frantisek Stampach},
      year={2022},
      eprint={2206.11007},
      archivePrefix={arXiv},
      primaryClass={math.SP}
}

@article {Giesselmann2009,
    AUTHOR = {Giesselmann, Jan},
     TITLE = {A convergence result for finite volume schemes on {R}iemannian
              manifolds},
   JOURNAL = {M2AN Math. Model. Numer. Anal.},
  FJOURNAL = {M2AN. Mathematical Modelling and Numerical Analysis},
    VOLUME = {43},
      YEAR = {2009},
    NUMBER = {5},
     PAGES = {929--955},
      ISSN = {0764-583X,1290-3841},
   MRCLASS = {35L65 (35L45 74S10)},
  MRNUMBER = {2559739},
MRREVIEWER = {Wen\ Shen},
       DOI = {10.1051/m2an/2009013},
       URL = {https://doi.org/10.1051/m2an/2009013},
}

@article {AlvarezBolte2008,
    AUTHOR = {Alvarez, F. and Bolte, J. and Munier, J.},
     TITLE = {A unifying local convergence result for {N}ewton's method in
              {R}iemannian manifolds},
   JOURNAL = {Found. Comput. Math.},
  FJOURNAL = {Foundations of Computational Mathematics. The Journal of the
              Society for the Foundations of Computational Mathematics},
    VOLUME = {8},
      YEAR = {2008},
    NUMBER = {2},
     PAGES = {197--226},
      ISSN = {1615-3375,1615-3383},
   MRCLASS = {65H10 (49M15 58C15 90C30 90C48)},
  MRNUMBER = {2407031},
MRREVIEWER = {Daniel\ J.\ Bates},
       DOI = {10.1007/s10208-006-0221-6},
       URL = {https://doi.org/10.1007/s10208-006-0221-6},
}

@article {Holst2001,
    AUTHOR = {Holst, M.},
     TITLE = {Adaptive numerical treatment of elliptic systems on manifolds},
   JOURNAL = {Adv. Comput. Math.},
  FJOURNAL = {Advances in Computational Mathematics},
    VOLUME = {15},
      YEAR = {2001},
    NUMBER = {1-4},
     PAGES = {139--191},
      ISSN = {1019-7168,1572-9044},
   MRCLASS = {65N30 (58J05 65N50 83-08)},
  MRNUMBER = {1887732},
       DOI = {10.1023/A:1014246117321},
       URL = {https://doi.org/10.1023/A:1014246117321},
}

@article {AmorimBenArtzi2005,
    AUTHOR = {Amorim, Paulo and Ben-Artzi, Matania and LeFloch, Philippe G.},
     TITLE = {Hyperbolic conservation laws on manifolds: total variation
              estimates and the finite volume method},
   JOURNAL = {Methods Appl. Anal.},
  FJOURNAL = {Methods and Applications of Analysis},
    VOLUME = {12},
      YEAR = {2005},
    NUMBER = {3},
     PAGES = {291--323},
      ISSN = {1073-2772,1945-0001},
   MRCLASS = {35L65 (58J45 76M12 76N10)},
  MRNUMBER = {2254012},
MRREVIEWER = {Alberto\ Valli},
       DOI = {10.4310/MAA.2005.v12.n3.a6},
       URL = {https://doi.org/10.4310/MAA.2005.v12.n3.a6},
}

@article {Fiori2017,
    AUTHOR = {Fiori, Simone},
     TITLE = {Nonlinear damped oscillators on {R}iemannian manifolds:
              numerical simulation},
   JOURNAL = {Commun. Nonlinear Sci. Numer. Simul.},
  FJOURNAL = {Communications in Nonlinear Science and Numerical Simulation},
    VOLUME = {47},
      YEAR = {2017},
     PAGES = {207--222},
      ISSN = {1007-5704,1878-7274},
   MRCLASS = {70G45 (34C26 34C40 35R01)},
  MRNUMBER = {3600269},
MRREVIEWER = {Oleg\ Gendelman},
       DOI = {10.1016/j.cnsns.2016.11.025},
       URL = {https://doi.org/10.1016/j.cnsns.2016.11.025},
}

@article {GarkeNurnberg2021,
    AUTHOR = {Garcke, Harald and N\"urnberg, Robert},
     TITLE = {Numerical approximation of boundary value problems for
              curvature flow and elastic flow in {R}iemannian manifolds},
   JOURNAL = {Numer. Math.},
  FJOURNAL = {Numerische Mathematik},
    VOLUME = {149},
      YEAR = {2021},
    NUMBER = {2},
     PAGES = {375--415},
      ISSN = {0029-599X,0945-3245},
   MRCLASS = {65M60 (35K55 53E99)},
  MRNUMBER = {4332794},
       DOI = {10.1007/s00211-021-01231-6},
       URL = {https://doi.org/10.1007/s00211-021-01231-6},
}

@article {CruzeiroMalliavin2006,
    AUTHOR = {Cruzeiro, A. B. and Malliavin, P.},
     TITLE = {Numerical approximation of diffusions in {$\Bbb R^d$} using
              normal charts of a {R}iemaannian manifold},
   JOURNAL = {Stochastic Process. Appl.},
  FJOURNAL = {Stochastic Processes and their Applications},
    VOLUME = {116},
      YEAR = {2006},
    NUMBER = {7},
     PAGES = {1088--1095},
      ISSN = {0304-4149,1879-209X},
   MRCLASS = {58J65 (60J60 65C30)},
  MRNUMBER = {2238615},
MRREVIEWER = {Emmanuel\ Gobet},
       DOI = {10.1016/j.spa.2006.02.004},
       URL = {https://doi.org/10.1016/j.spa.2006.02.004},
}

@article {BarettGarke2019,
    AUTHOR = {Barrett, John W. and Garcke, Harald and N\"urnberg, Robert},
     TITLE = {Stable discretizations of elastic flow in {R}iemannian
              manifolds},
   JOURNAL = {SIAM J. Numer. Anal.},
  FJOURNAL = {SIAM Journal on Numerical Analysis},
    VOLUME = {57},
      YEAR = {2019},
    NUMBER = {4},
     PAGES = {1987--2018},
      ISSN = {0036-1429,1095-7170},
   MRCLASS = {65M60 (35K55 53A30 53C44 65M12)},
  MRNUMBER = {3992503},
       DOI = {10.1137/18M1227111},
       URL = {https://doi.org/10.1137/18M1227111},
}

@article {FioriCervigni2022,
    AUTHOR = {Fiori, Simone and Cervigni, Italo and Ippoliti, Mattia and
              Menotta, Claudio},
     TITLE = {Synthetic nonlinear second-order oscillators on {R}iemannian
              manifolds and their numerical simulation},
   JOURNAL = {Discrete Contin. Dyn. Syst. Ser. B},
  FJOURNAL = {Discrete and Continuous Dynamical Systems. Series B. A Journal
              Bridging Mathematics and Sciences},
    VOLUME = {27},
      YEAR = {2022},
 NUMBER = {3},
     PAGES = {1227--1262},
      ISSN = {1531-3492,1553-524X},
   MRCLASS = {58J99 (34C40 37M05 65L05 65P99)},
  MRNUMBER = {4385789},
MRREVIEWER = {Giovanni\ Rastelli},
       DOI = {10.3934/dcdsb.2021088},
       URL = {https://doi.org/10.3934/dcdsb.2021088},
}

@inproceedings{YuDeSa2019,
 author = {Yu, Tao and De Sa, Christopher M},
 booktitle = {Advances in Neural Information Processing Systems},
 editor = {H. Wallach and H. Larochelle and A. Beygelzimer and F. d\textquotesingle Alch\'{e}-Buc and E. Fox and R. Garnett},
 pages = {},
 publisher = {Curran Associates, Inc.},
 title = {Numerically Accurate Hyperbolic Embeddings Using Tiling-Based Models},
 url = {https://proceedings.neurips.cc/paper_files/paper/2019/file/82c2559140b95ccda9c6ca4a8b981f1e-Paper.pdf},
 volume = {32},
 year = {2019}
}

@article {WuSun2004,
    AUTHOR = {Wu, Xiaonan and Sun, Zhi-Zhong},
     TITLE = {Convergence of difference scheme for heat equation in
              unbounded domains using artificial boundary conditions},
   JOURNAL = {Appl. Numer. Math.},
  FJOURNAL = {Applied Numerical Mathematics. An IMACS Journal},
    VOLUME = {50},
      YEAR = {2004},
    NUMBER = {2},
     PAGES = {261--277},
      ISSN = {0168-9274,1873-5460},
   MRCLASS = {65M06 (65M12)},
  MRNUMBER = {2066740},
       DOI = {10.1016/j.apnum.2004.01.001},
       URL = {https://doi.org/10.1016/j.apnum.2004.01.001},
}

@article {PangYang2017,
    AUTHOR = {Pang, Gang and Yang, Yibo and Tang, Shaoqiang},
     TITLE = {Exact boundary condition for semi-discretized
              {S}chr\"{o}dinger equation and heat equation in a rectangular
              domain},
   JOURNAL = {J. Sci. Comput.},
  FJOURNAL = {Journal of Scientific Computing},
    VOLUME = {72},
      YEAR = {2017},
    NUMBER = {1},
     PAGES = {1--13},
      ISSN = {0885-7474,1573-7691},
   MRCLASS = {65M06 (35Q55 65M12)},
  MRNUMBER = {3661095},
MRREVIEWER = {Tore\ Fl\aa tten},
       DOI = {10.1007/s10915-016-0344-0},
       URL = {https://doi.org/10.1007/s10915-016-0344-0},
}

@article {ZhengXie2023,
    AUTHOR = {Zheng, Chunxiong and Xie, Jiangming},
     TITLE = {Fast artificial boundary method for the heat equation on
              unbounded domains with strip tails},
   JOURNAL = {J. Comput. Appl. Math.},
  FJOURNAL = {Journal of Computational and Applied Mathematics},
    VOLUME = {425},
      YEAR = {2023},
     PAGES = {Paper No. 115032, 17},
      ISSN = {0377-0427,1879-1778},
   MRCLASS = {65M60 (35K05)},
  MRNUMBER = {4530829},
       DOI = {10.1016/j.cam.2022.115032},
       URL = {https://doi.org/10.1016/j.cam.2022.115032},
}

@article {BerchioGanguly2020,
    AUTHOR = {Berchio, Elvise and Ganguly, Debdip and Grillo, Gabriele and
              Pinchover, Yehuda},
     TITLE = {An optimal improvement for the {H}ardy inequality on the
              hyperbolic space and related manifolds},
   JOURNAL = {Proc. Roy. Soc. Edinburgh Sect. A},
  FJOURNAL = {Proceedings of the Royal Society of Edinburgh. Section A.
              Mathematics},
    VOLUME = {150},
      YEAR = {2020},
    NUMBER = {4},
     PAGES = {1699--1736},
      ISSN = {0308-2105,1473-7124},
   MRCLASS = {46E35 (26D10 31C12 58J60)},
  MRNUMBER = {4122432},
MRREVIEWER = {Adamaria\ Perrotta},
       DOI = {10.1017/prm.2018.139},
       URL = {https://doi.org/10.1017/prm.2018.139},
}

@article {Kristaly2022,
    AUTHOR = {Krist\'aly, Alexandru},
     TITLE = {New features of the first eigenvalue on negatively curved
              spaces},
   JOURNAL = {Adv. Calc. Var.},
  FJOURNAL = {Advances in Calculus of Variations},
    VOLUME = {15},
      YEAR = {2022},
    NUMBER = {3},
     PAGES = {475--495},
      ISSN = {1864-8258,1864-8266},
   MRCLASS = {35P15 (33C05 49R05 58B20 58C40)},
  MRNUMBER = {4451902},
       DOI = {10.1515/acv-2019-0103},
       URL = {https://doi.org/10.1515/acv-2019-0103},
}

@incollection {Pesenson2009,
    AUTHOR = {Pesenson, Isaac},
     TITLE = {A discrete {H}elgason-{F}ourier transform for {S}obolev and
              {B}esov functions on noncompact symmetric spaces},
 BOOKTITLE = {Radon transforms, geometry, and wavelets},
    SERIES = {Contemp. Math.},
    VOLUME = {464},
     PAGES = {231--247},
 PUBLISHER = {Amer. Math. Soc., Providence, RI},
      YEAR = {2008},
      ISBN = {978-0-8218-4327-7},
   MRCLASS = {43A85 (22E30 42B35 46E35)},
  MRNUMBER = {2440139},
MRREVIEWER = {Leszek\ Skrzypczak},
       DOI = {10.1090/conm/464/09087},
       URL = {https://doi.org/10.1090/conm/464/09087},
}

@article {BuragoIvanov2014,
    AUTHOR = {Burago, Dmitri and Ivanov, Sergei and Kurylev, Yaroslav},
     TITLE = {A graph discretization of the {L}aplace-{B}eltrami operator},
   JOURNAL = {J. Spectr. Theory},
  FJOURNAL = {Journal of Spectral Theory},
    VOLUME = {4},
      YEAR = {2014},
    NUMBER = {4},
     PAGES = {675--714},
      ISSN = {1664-039X,1664-0403},
   MRCLASS = {58J50 (05C50 53C21 58J60 65N25)},
  MRNUMBER = {3299811},
MRREVIEWER = {Emil\ Saucan},
       DOI = {10.4171/JST/83},
       URL = {https://doi.org/10.4171/JST/83},
}

@incollection {MugelliTalenti1997,
    AUTHOR = {Mugelli, Francesco and Talenti, Giorgio},
     TITLE = {Sobolev inequalities in {$2$}-dimensional hyperbolic space},
 BOOKTITLE = {General inequalities, 7 ({O}berwolfach, 1995)},
    SERIES = {Internat. Ser. Numer. Math.},
    VOLUME = {123},
     PAGES = {201--216},
 PUBLISHER = {Birkh\"auser, Basel},
      YEAR = {1997},
      ISBN = {3-7643-5722-3},
   MRCLASS = {46E35 (30F45)},
  MRNUMBER = {1457280},
MRREVIEWER = {W.\ P.\ Ziemer},
}

@article {MugelliTalenti1998,
    AUTHOR = {Mugelli, Francesco and Talenti, Giorgio},
     TITLE = {Sobolev inequalities in {$2$}-{D} hyperbolic space: a
              borderline case},
   JOURNAL = {J. Inequal. Appl.},
  FJOURNAL = {Journal of Inequalities and Applications},
    VOLUME = {2},
      YEAR = {1998},
    NUMBER = {3},
     PAGES = {195--228},
      ISSN = {1025-5834,1029-242X},
   MRCLASS = {46E35 (26D10 34C99)},
  MRNUMBER = {1671679},
MRREVIEWER = {Alexander\ Brudnyi},
       DOI = {10.1155/S1025583498000125},
       URL = {https://doi.org/10.1155/S1025583498000125},
}

@article {Vazquez2022,
    AUTHOR = {V\'azquez, Juan Luis},
     TITLE = {Asymptotic behaviour for the heat equation in hyperbolic
              space},
   JOURNAL = {Comm. Anal. Geom.},
  FJOURNAL = {Communications in Analysis and Geometry},
    VOLUME = {30},
      YEAR = {2022},
    NUMBER = {9},
     PAGES = {2123--2156},
      ISSN = {1019-8385,1944-9992},
   MRCLASS = {58J35 (58J37)},
  MRNUMBER = {4631177},
}

@article {NgoNguyen2019,
    AUTHOR = {Ng\^o, Qu\cfac oc Anh and Nguyen, Van Hoang},
     TITLE = {Sharp constant for {P}oincar\'e-type inequalities in the
              hyperbolic space},
   JOURNAL = {Acta Math. Vietnam.},
  FJOURNAL = {Acta Mathematica Vietnamica},
    VOLUME = {44},
      YEAR = {2019},
    NUMBER = {3},
     PAGES = {781--795},
      ISSN = {0251-4184,2315-4144},
   MRCLASS = {26D10 (31C12 46E35)},
  MRNUMBER = {3988209},
MRREVIEWER = {Aigerim\ Kalybay},
       DOI = {10.1007/s40306-018-0269-9},
       URL = {https://doi.org/10.1007/s40306-018-0269-9},
}

@article {AnkerPapageorgiou2023,
    AUTHOR = {Anker, Jean-Philippe and Papageorgiou, Effie and Zhang,
              Hong-Wei},
     TITLE = {Asymptotic behavior of solutions to the heat equation on
              noncompact symmetric spaces},
   JOURNAL = {J. Funct. Anal.},
  FJOURNAL = {Journal of Functional Analysis},
    VOLUME = {284},
      YEAR = {2023},
    NUMBER = {6},
     PAGES = {Paper No. 109828, 43},
      ISSN = {0022-1236,1096-0783},
   MRCLASS = {22E30 (35B40 35K05 58J35)},
  MRNUMBER = {4530900},
MRREVIEWER = {Rudra\ P.\ Sarkar},
       DOI = {10.1016/j.jfa.2022.109828},
       URL = {https://doi.org/10.1016/j.jfa.2022.109828},
}

@article {BandleGonzalez2018,
    AUTHOR = {Bandle, Catherine and Gonz\'alez, Mar\'ia del Mar and
              Fontelos, Marco A. and Wolanski, Noemi},
     TITLE = {A nonlocal diffusion problem on manifolds},
   JOURNAL = {Comm. Partial Differential Equations},
  FJOURNAL = {Communications in Partial Differential Equations},
    VOLUME = {43},
      YEAR = {2018},
    NUMBER = {4},
     PAGES = {652--676},
      ISSN = {0360-5302,1532-4133},
   MRCLASS = {58J40 (35K15 35R01 35R03 45K05 47G10)},
  MRNUMBER = {3902174},
MRREVIEWER = {Jean-Marc\ Delort},
       DOI = {10.1080/03605302.2018.1459685},
       URL = {https://doi.org/10.1080/03605302.2018.1459685},
}

@article {Banica2007,
    AUTHOR = {Banica, V.},
     TITLE = {The nonlinear {S}chr\"odinger equation on hyperbolic space},
   JOURNAL = {Comm. Partial Differential Equations},
  FJOURNAL = {Communications in Partial Differential Equations},
    VOLUME = {32},
      YEAR = {2007},
    NUMBER = {10-12},
     PAGES = {1643--1677},
      ISSN = {0360-5302,1532-4133},
   MRCLASS = {35Q55 (35B40 58J35)},
  MRNUMBER = {2372482},
MRREVIEWER = {Yu.\ E.\ Gliklikh},
       DOI = {10.1080/03605300600854332},
       URL = {https://doi.org/10.1080/03605300600854332},
}

@article {Tataru2001,
    AUTHOR = {Tataru, Daniel},
     TITLE = {Strichartz estimates in the hyperbolic space and global
              existence for the semilinear wave equation},
   JOURNAL = {Trans. Amer. Math. Soc.},
  FJOURNAL = {Transactions of the American Mathematical Society},
    VOLUME = {353},
      YEAR = {2001},
    NUMBER = {2},
     PAGES = {795--807},
      ISSN = {0002-9947,1088-6850},
   MRCLASS = {35L70 (35L05 58J45)},
  MRNUMBER = {1804518},
MRREVIEWER = {Nickolai\ A.\ Lar\cprime kin},
       DOI = {10.1090/S0002-9947-00-02750-1},
       URL = {https://doi.org/10.1090/S0002-9947-00-02750-1},
}

@book {Terras2013,
    AUTHOR = {Terras, Audrey},
     TITLE = {Harmonic analysis on symmetric spaces---{E}uclidean space, the
              sphere, and the {P}oincar\'e{} upper half-plane},
   EDITION = {Second},
 PUBLISHER = {Springer, New York},
      YEAR = {2013},
     PAGES = {xviii+413},
      ISBN = {978-1-4614-7971-0; 978-1-4614-7972-7},
   MRCLASS = {22E30 (11F72 43-02 43A85)},
  MRNUMBER = {3100414},
MRREVIEWER = {Rudra\ P.\ Sarkar},
       DOI = {10.1007/978-1-4614-7972-7},
       URL = {https://doi.org/10.1007/978-1-4614-7972-7},
}

@book {Helgason2008,
    AUTHOR = {Helgason, Sigurdur},
     TITLE = {Geometric analysis on symmetric spaces},
    SERIES = {Mathematical Surveys and Monographs},
    VOLUME = {39},
   EDITION = {Second},
 PUBLISHER = {American Mathematical Society, Providence, RI},
      YEAR = {2008},
     PAGES = {xviii+637},
      ISBN = {978-0-8218-4530-1},
   MRCLASS = {22E46 (22E30 43A85 53C35 58J60)},
  MRNUMBER = {2463854},
MRREVIEWER = {Jacques\ Faraut},
       DOI = {10.1090/surv/039},
       URL = {https://doi.org/10.1090/surv/039},
}

@misc{loustau2021hyperbolicgeometry,
      title={Hyperbolic geometry}, 
      author={Brice Loustau},
      year={2020},
      eprint={2003.11180},
      archivePrefix={arXiv},
      primaryClass={math.DG},
      url={https://arxiv.org/abs/2003.11180v2}, 
}

@article {IgnatManeaMoroianu2024,
    AUTHOR = {Gonz\'alez, Mar\'ia del Mar and Ignat, Liviu I. and Manea,
              Drago\c s{} and Moroianu, Sergiu},
     TITLE = {Concentration limit for non-local dissipative
              convection--diffusion kernels on the hyperbolic space},
   JOURNAL = {Nonlinear Anal.},
  FJOURNAL = {Nonlinear Analysis. Theory, Methods \& Applications. An
              International Multidisciplinary Journal},
    VOLUME = {248},
      YEAR = {2024},
     PAGES = {Paper No. 113618},
      ISSN = {0362-546X,1873-5215},
   MRCLASS = {45K05 (45M05 58J35)},
  MRNUMBER = {4781094},
       DOI = {10.1016/j.na.2024.113618},
       URL = {https://doi.org/10.1016/j.na.2024.113618},
}

@article {beci-yehuda,
    AUTHOR = {Keller, Matthias and Pinchover, Yehuda and Pogorzelski, Felix},
     TITLE = {An improved discrete {H}ardy inequality},
   JOURNAL = {Amer. Math. Monthly},
  FJOURNAL = {American Mathematical Monthly},
    VOLUME = {125},
      YEAR = {2018},
    NUMBER = {4},
     PAGES = {347--350},
      ISSN = {0002-9890,1930-0972},
   MRCLASS = {26D15},
  MRNUMBER = {3779222},
MRREVIEWER = {Aingeru\ Fern\'andez\ Bertolin},
       DOI = {10.1080/00029890.2018.1420995},
       URL = {https://doi.org/10.1080/00029890.2018.1420995},
}

@article {pescaru-beci,
    AUTHOR = {Fischer, Florian and Keller, Matthias and Pogorzelski, Felix},
     TITLE = {An improved discrete {$p$}-{H}ardy inequality},
   JOURNAL = {Integral Equations Operator Theory},
  FJOURNAL = {Integral Equations and Operator Theory},
    VOLUME = {95},
      YEAR = {2023},
    NUMBER = {4},
     PAGES = {Paper No. 24, 17},
      ISSN = {0378-620X,1420-8989},
   MRCLASS = {26D15 (47J05)},
  MRNUMBER = {4648598},
MRREVIEWER = {Vicen\c tiu\ D.\ R\u adulescu},
       DOI = {10.1007/s00020-023-02743-6},
       URL = {https://doi.org/10.1007/s00020-023-02743-6},
}

@article {david-stampac,
    AUTHOR = {Krej\v ci\v r\'ik, David and \v Stampach, Franti\v sek},
     TITLE = {A sharp form of the discrete {H}ardy inequality and the
              {K}eller-{P}inchover-{P}ogorzelski inequality},
   JOURNAL = {Amer. Math. Monthly},
  FJOURNAL = {American Mathematical Monthly},
    VOLUME = {129},
      YEAR = {2022},
    NUMBER = {3},
     PAGES = {281--283},
      ISSN = {0002-9890,1930-0972},
   MRCLASS = {26D15},
  MRNUMBER = {4399059},
MRREVIEWER = {Ivan\ Gadjev},
       DOI = {10.1080/00029890.2022.2011569},
       URL = {https://doi.org/10.1080/00029890.2022.2011569},
}

@article{Jiang2024,
title = {Generalized finite difference method on unknown manifolds},
journal = {Journal of Computational Physics},
volume = {502},
pages = {112812},
year = {2024},
issn = {0021-9991},
doi = {https://doi.org/10.1016/j.jcp.2024.112812},
url = {https://www.sciencedirect.com/science/article/pii/S0021999124000615},
author = {Shixiao Willing Jiang and Rongji Li and Qile Yan and John Harlim},
}

@article{Izmestiev2025,
author = {Izmestiev, Ivan and Lam, Wai Yeung},
title = {Discrete Laplacians — Spherical and hyperbolic},
journal = {Journal of the London Mathematical Society},
volume = {112},
number = {1},
pages = {e70235},
doi = {https://doi.org/10.1112/jlms.70235},
url = {https://londmathsoc.onlinelibrary.wiley.com/doi/abs/10.1112/jlms.70235},
eprint = {https://londmathsoc.onlinelibrary.wiley.com/doi/pdf/10.1112/jlms.70235},
year = {2025}
}
\end{document}